\documentclass{amsart}
\usepackage{graphicx}
\usepackage{amscd}
\usepackage{amsmath}
\usepackage{amsfonts}
\usepackage{amssymb}
\usepackage{setspace}
\usepackage{enumerate}         
\usepackage{color}
\usepackage{url}
\usepackage{amsthm}
\usepackage{hyperref}
\usepackage{bm}
\usepackage{xy}
\usepackage{enumitem}

\theoremstyle{plain}
\newtheorem{theorem}{Theorem}[section]

\newtheorem{corollary}[theorem]{Corollary}

\newtheorem{lemma}[theorem]{Lemma}

\newtheorem{proposition}[theorem]{Proposition}

\newtheorem*{assumption*}{Assumption}
\theoremstyle{remark}
\newtheorem{remark}[theorem]{Remark}

\numberwithin{equation}{section}

\newcommand{\ind}{1\!\kern-1pt \mathrm{I}}
\newcommand{\rsto}{]\!\kern-1.8pt ]}
\newcommand{\lsto}{[\!\kern-1.7pt [}

\numberwithin{equation}{section}

\newcommand{\id}{\operatorname{id}}

\newcommand{\im}{\ensuremath{\mathsf{i}}}

\begin{document}
\title[Fourier transform methods for pathwise covariance estimation]{Fourier transform methods for pathwise covariance estimation in the presence of jumps}
\author{Christa Cuchiero and Josef Teichmann}
\address{Vienna University of Technology, Wiedner Hauptstrasse 8 / 105-1, A-1040 Wien, Austria,
\newline ETH Z\"urich, D-Math, R\"amistrasse 101, CH-8092 Z\"urich, Switzerland}
\email{cuchiero@fam.tuwien.ac.at, jteichma@math.ethz.ch}
\thanks{The first author gratefully acknowledges the partial financial support from the European Research Council (ERC, grant FA506041). The first and second author gratefully acknowledge the support from the ETH-foundation.}

\begin{abstract}
We provide a new non-parametric Fourier procedure to estimate the trajectory of the instantaneous covariance process (from discrete observations of a multidimensional price process) in the presence of jumps extending the seminal work Malliavin and Mancino~\cite{MM:02, MM:09}. Our approach relies on a modification of (classical) jump-robust estimators of integrated realized covariance to estimate the Fourier coefficients of the covariance trajectory. Using Fourier-F\'ejer inversion we reconstruct the path of the instantaneous covariance. We prove consistency and central limit theorem (CLT) and in particular that the asymptotic estimator variance is smaller by a factor $ 2/3$ in comparison to classical local estimators.

The procedure is robust enough to allow for an iteration and we can show theoretically and empirically how to estimate the integrated realized covariance of the instantaneous stochastic covariance process. We apply these techniques to robust calibration problems for multivariate modeling in finance, i.e., the selection of a pricing measure by using time series and derivatives' price information simultaneously.

\end{abstract}


\maketitle

\section{Introduction}

The recent difficulties in the banking and insurance industry are to some extent due to insufficient modeling of multivariate stochastic phenomena which appear in financial markets. There are several reasons why modeling is insufficient, but the two most important ones are the following: first, realistic multivariate models are difficult to calibrate to market information due to a lack of analytic tractability, hence oversimplified models are in use in delicate multivariate situations, and, second, usually either time series data or derivatives' prices are used to select a model from a given model class but \emph{not both} sorts of available information simultaneously.
It is often argued that due to the difference of the statistical measure and the pricing measure we are actually not able to use the information simultaneously, except if we determine the statistical measure and make an ansatz for the market price of risk. Robust model calibration instead uses time series and option price information simultaneously without conjecturing about quantities which are as hard to identify as drifts.

\subsection{Robust Calibration}\label{sec:1}

We aim to develop methods which allow for \emph{robust calibration}, i.e., estimation and calibration of a model in a well specified sense simultaneously from time series and derivatives' prices data in order to select \emph{a pricing measure}. 
Reasons why both kinds of data should enter the field of model selection in mathematical finance are the high dimensional parameter space of multivariate models and the lack of liquidly traded multi-asset options, which makes a calibration procedure solely based on derivatives' data infeasible. This difficulty can be tackled by additionally using time series of asset prices, from which certain model parameters can be determined. It is useful to demonstrate what we actually mean with \emph{robust calibration} by means of an example: take a Heston model
\begin{align*}
d X_t &= \kappa ( \theta - X_t) dt + \sigma \sqrt{X_t} dZ_t \, , \\
d Y_t &= \mu- \frac{X_t}{2} \, dt + \sqrt{X_t} d B_t \, ,
\end{align*}
where $ X $ denotes the stochastic variance process and $ Y $ the logarithmic price of a stock. The model is written with respect to a pricing measure, i.e., $ \exp(Y) $ is a martingale, if $\mu=0$, otherwise the model is written with respect to the real world measure. Through robust calibration we have to identify the initial value $X_0$, $ Y_0 $, the parameters $ \kappa, \theta, \sigma $ and the correlation parameter $\rho$ between the Brownian motions $Z$ and $B$ in order to specify the model for purposes of pricing, hedging or short term risk management. If we specify additionally $\mu$ we can use the model for (long term) risk management.
  
Apparently at least the initial values $ X_0 $, $ Y_0 $, and the parameters $ \sigma $ and $ \rho $ do not change under equivalent measure changes, so in principal the parameters $X_0,Y_0$, $\sigma, \rho$ \emph{can} be identified from the observation of a single trajectory, and it does not matter with respect to which equivalent measure we observe this trajectory. On the other hand market implied values for those parameters should coincide with values estimated from the time series if the model is close to correct. Here ``market implied values'' means to choose model parameter values such that the model's derivatives' prices and the market prices coincide as good as possible. The other parameters $ \kappa $, $\theta $ and $ \mu $ can be changed under equivalent measure changes (given that we stay in the above parametrized class) and there values depend on the data, which are used to estimate them.

Formally speaking we have defined the above model on a filtered probability space $ (\Omega,\mathcal{F},\mathbb{P}) $ and we consider equivalent measures $ \mathbb{Q} \sim \mathbb{P} $. Having specified a set of parameters $ \Theta $ and a semimartingale $ S^\theta $ depending on parameters $ \theta \in \Theta $, we can then define an equivalence relationship, namely $ \theta_1 \sim \theta_2 $ if $ S^{\theta_1}_\ast \mathbb{P} = S^{\theta_2}_\ast \mathbb{Q} $, i.e., equality for the respective measures on the canonical probability space of c\`adl\`ag paths, for some $ \mathbb{Q} \sim \mathbb{P} $. This equivalence relation defines orbits on $ \Theta $ and the space of orbits $ \Theta/\sim $, where the latter set is the set of invariant parameters, i.e., those parameters of $ S^\theta $, which remain unchanged by equivalent measure changes. The set $\Theta/\sim$ is non-trivial if there are pathwise defined (estimator) functionals which determine parameters of the vector $\theta$. 

Having this basic stochastic fact in mind, it should be clear that non-parametric estimation of instantaneous covariance processes is \emph{the} important task to be performed, since it yields -- in the previous concrete case -- information on $ \sigma $, $ \rho $ and the trajectory of instantaneous stochastic variance $ t \mapsto X_t $ along the observation interval. The parameters $ \kappa $ and $ \theta $ will rather be calibrated from derivatives' prices, since they cannot be identified from time series information without additional assumptions. The parameter $\mu$ cannot be identified from derivatives' prices but only from time series data and is hardest to identify. Identifying $ t \mapsto X_t $, $ \sigma $ and $ \rho $ from time series data and $ \kappa $, $\theta $ from derivatives' is what we would call \emph{robust calibration}: the procedure is more robust than classical calibration (using only derivatives' data), since more data are used, and it allows for model rejection if calibration is not feasible after identifying invariant parameters.


To be precise on what we mean by time series data: we think of intraday (minute) price data for liquid instruments along periods of months up to years, such that we have about $ 10^5 $ data points available without loosing assumptions on time-homogeneity.

\subsection{Program of the Article and Related Literature  on Non-Parametric Covariance Estimation}
Based on the above calibration concept, the goal of the present article is to find methods which allow to estimate \emph{non-parametrically} in a multivariate setting the stochastic covariance of the stochastic covariance process $X$, and, to estimate the stochastic correlation between the log-price process $Y$ and the stochastic covariance process $X$. This involves a two step procedure where we first need to recover the realized path of the instantaneous covariance, from which we can then estimate the second order quantities. 

In order to achieve the first task of pathwise covariance reconstruction, we combine jump robust estimators with instantaneous covariance estimation based on Fourier methods (see Malliavin and Mancino~\cite{MM:02, MM:09}). More precisely, we modify jump robust estimators of integrated realized covariance (as considered by~\cite{BNGJPS:06,BNSW:06,J:08,PV:09,TT:11,W:06}) to obtain estimators for the Fourier coefficients of the realized path of the instantaneous covariance. By means of Fourier-F\'ejer inversion we then get an estimator for the instantaneous realized covariance. For this estimator we prove consistency and a central limit theorem, showing that the asymptotic estimator variance is smaller by a factor $2/3$ in comparison to the classical instantaneous covariance estimator. For the subsequent estimation of the second order quantities we then rely on existing jump robust estimators for integrated covariance, into which we plug the reconstructed path of the covariance. For this estimator of the integrated covariance of the covariance process we also provide a central limit theorem.

Concerning the literature on non-parametric covariance estimation in the presence of jumps, there are many precise asymptotic results on (i) \emph{integrated realized covariance} estimation available, but only few (e.g.~for classical sum of squares estimators) in the case of (ii) \emph{instantaneous covariance}.  To the best of our knowledge there are no asymptotic results on the estimation of the above described (iii) \emph{second order quantities} available when the involved processes have jumps.\\
With regard to (i), non-parametric jump robust estimators for the \emph{integrated realized covariance} range from threshold methods as considered by Mancini~\cite{M:09, M:11} to  (Bi- and Multi-)Power-variation estimators as studied by Barndorff-Nielsen et al.~\cite{BNS:04,BNGJPS:06, BNGS:04,BNSW:06}. These latter estimators have been successively generalized by replacing the power function with different specifications (see, e.g., Jacod~\cite{J:08}, Podolskij~\cite{P:06} and Tauchen and Todorov~\cite{TT:11}). An excellent account of all kinds of integrated covariance estimators and their asymptotic properties can be found in the book of Jacod and Protter~\cite{JP:12} and the literature therein. \\
Concerning non-parametric techniques to measure (ii) the \emph{instantaneous covariance}, the majority of the proposed estimators is based on differentiation of the integrated variance as for example in  Alvarez et al.~\cite{APPS:12},  Bandi and Ren\`o~\cite{BR:11}, Mykland and Zhang~\cite{MZ:08} or Jacod and Protter~\cite[Section 13.3]{JP:12}. These estimators correspond to so-called \emph{local realized variance estimators}. A similar more general approach relies on kernel estimators as in Fan and Wang~\cite{FW:08} or Kristensen~\cite{K:10}. The above described classical local realized variance estimators is a particular case of these kernel estimators, corresponding to the choice of the uniform kernel. Another strand of literature is based on Fourier methods as in Malliavin and Mancino~\cite{MM:02, MM:09}
and wavelet analysis as in Genon-Catalot et al.~\cite{GC:92}. With the exception of~\cite{BR:11} and~\cite{JP:12}, who consider the classical and truncated local realized variance estimator in the presence of jumps, the common assumption of the above articles is \emph{continuity of the trajectories} of the (log)-price process.\\
Non-parametric estimation of (iii) the \emph{second order quantities} based on the local realized variance estimator have already been considered by Vetter~\cite{V:12} and Wang and Mykland~\cite{MW:12} for the estimation of the correlation between the log-price and the variance process. Barucci and Mancino~\cite{BM:10} provide alternative estimation techniques based on Fourier methods for both, the variance of the variance process and the correlation between the log-price and the variance process. While the setting for all these estimators is based on the assumption of It\^o-processes with \emph{continuous} trajectories, we establish estimation procedures which also work in the presence of jumps. This is particularly interesting for stochastic volatility of volatility models with jump components. Recently, a new class of this type of models has been introduced by Barndorff-Nielsen and Veraart~\cite{BNV:13}.
As outlined in~\cite[Section 2.5]{BNV:13}, model identification, which means in this context testing whether an additional volatility component is present or not, should be based on the estimation of the quadratic variation of the spot volatility and thus requires jump robust estimators of the second order quantities.

The remainder of the article is organized as follows. In Section~\ref{sec:setting} we introduce the assumptions on the log-price and the instantaneous covariance process and Section~\ref{sec:4} gives an overview of the different steps in our estimation procedure. Section~\ref{sec:main_theorem} contains the statements of the main theorems. The remaining sections are dedicated to the proofs of the main theorems. In Section~\ref{sec:FC} we consider asymptotic properties for jump robust estimators of the Fourier coefficients, while in Section~\ref{sec:spotvar} consistency and a central limit theorem are shown for the Fourier-F\'ejer instantaneous covariance estimator. Section~\ref{sec:covcov} concludes with the proof of asymptotic normality for estimators of the integrated covariance of the instantaneous stochastic covariance process. The Appendix contains a simulation study~\ref{sec:sim} illustrating our theoretical findings.

\section{Setting and Methodology Overview}\label{sec:setting}

\subsection{Setting and Assumptions}

Throughout we let $T >0$ be fixed and work on a filtered probability space $(\Omega, \mathcal{F}, (\mathcal{F}_t)_{0  \leq t \leq T}, \mathbb{P})$, 
where we consider a $d$-dimensional (discounted) asset price process $(S_t)_{0 \leq t \leq T}$, which is supposed to be positive componentwise, and
adapted to the filtration $(\mathcal{F}_t)_{0 \leq t \leq T}$.
Due to positivity of $S$ we further assume 
\[
 S_t=(\exp({Y_{t,1}}), \ldots,\exp({Y_{t,d}}))^{\top}, \quad 0 \leq t \leq T,
\]
where $(Y_t)_{0 \leq t \leq T}$ denotes the $d$-dimensional (discounted)  logarithmic price process starting at $Y_0=y \in \mathbb{R}^d$ a.s. Due to no-arbitrage consideration $S$ and thus also $Y$ are supposed to be semimartingales with a rich structure of jumps.

Furthermore, let us introduce some mild structural assumptions on the log-price process $Y$,
namely that it is an It\^o-semimartingale of the following form:

\begin{assumption*}[H]
The logarithmic price process $Y$ satisfies
\begin{equation}
\begin{split}~\label{eq:Y}
Y_t&=y+\int_0^t b^Y_s ds + \int_0^t\sqrt{X_{s-}} dZ_s + \int_0^t \int_{\mathbb{R}^d} \chi(\xi)  (\mu^Y(d\xi,ds)- K_{s}(d\xi)ds) \\
&\quad+\int_0^t \int_{\mathbb{R}^d} (\xi-\chi(\xi))  \mu^Y(d\xi,ds) ,
\end{split}
\end{equation}
where $\sqrt{X}$ denotes the unique matrix square root on $S_d^+$, the space of positive semidefinite matrices, and
\begin{itemize}
\item $Z$ is a $d$-dimensional Brownian motion,
\item $b^Y$ an $\mathbb{R}^{d}$-valued locally bounded process,
\item $X$ a c\`adl\`ag process taking values in $S_d^+$ and
\item $\mu^Y(d\xi,dt)$ the random measure associated with the jumps of $Y$, whose compensator is given by $K_t(d\xi)dt$, where $K_t(d\xi)=K_t(\omega,d\xi)$ is for each $(\omega,t)$ a measure on $\mathbb{R}^d$. 
\end{itemize}
\end{assumption*}

\begin{remark}
\begin{enumerate}
\item
Usually, e.g., in~\cite{JP:12}, the assumption of an It\^o-semimartingale is formulated in terms of the \emph{Gringelions representation}, which means that there exists
an extension of the probability space, on which are defined a $d'$-dimensional Brownian motion $W$ and a Poisson random measure $\mathfrak{p}$ with L\'evy measure $\lambda$ such that
\begin{align*}
Y_t&=y+\int_0^t b^Y_s ds + \int_0^t\sigma_s dW_s + \int_0^t \int_{\mathbb{R}^d} \delta(\xi,s-) 1_{\{\|\delta\|\leq 1\}}  (\mathfrak{p}(d\xi,ds)- \lambda(d\xi)ds) \\
&\quad+\int_0^t \int_{\mathbb{R}^d} \delta(\xi,s-) 1_{\{\|\delta\|>1\}} \mathfrak{p}(d\xi,ds) ,
\end{align*}
where $\sigma_t$ is an $\mathbb{R}^{d\times d'}$-valued predictable process such that $\sigma^{\top}\sigma=X$ and $\delta$ is a predictable $\mathbb{R}^d$-valued function on $\Omega\times \mathbb{R}^d \times \mathbb{R}_+$.
In view of applications to (affine) processes, whose characteristics are given in terms of representation~\eqref{eq:Y}, we 
prefer to use the formulation of Assumption $(H)$ and do not use the Gringelions representation, since it involves an extension of the probability space and the specific form of $\delta$ and $\mathfrak{p}$ is not evident.
\item The conditions on the characteristics in Assumption $(H)$ correspond to those of~\cite[Assumption 4.4.2 (or $(H)$)]{JP:12} with the only difference that for the moment we do not assume some kind of local integrability on the jump measure, as in (iii) of~\cite[Assumption 4.4.2]{JP:12}.
\end{enumerate}
\end{remark}

Note that the assumption of an It\^o-semimartingale is satisfied by all continuous-time models used in mathematical finance. Indeed, beside the assumption of \emph{absolutely continuous} characteristics, this is the most general model-free setting which is 
in accordance with the no-arbitrage paradigm. This assumption is actually also the only one needed to prove
consistency of the Fourier-F\'ejer instantaneous covariance estimator, denoted by $\widehat{X}_t$. However for establishing a central limit theorem we also need some structural assumptions on the instantaneous covariance process $X$:

\begin{assumption*}[H1]
Assumption $(H)$ holds and the instantaneous covariance process $X$ 
is an It\^o-semimartingale of the form
\begin{align*}
 X_t=x +\int_0^t b^X_s ds +\sum_{j=1}^p \int_0^t Q_{s}^j dB_{s,j} +\int_0^t \int_{S_d} \xi  \mu^X(d\xi,ds)
\end{align*}
where
\begin{itemize}
\item $B$ is $p$-dimensional Brownian motion, which can be correlated with $Z$, the Brownian motion driving the log-price process,
such that $d\langle Z_i, B_j\rangle_t=\rho_{t,ij}dt$, where $\rho_{ij}$ is adapted c\`agl\`ad for all $i\in \{1,\ldots,d\}$ and 
$j \in \{1,\ldots,p\}$,
\item $b^X$ is an $\mathbb{R}^{d\times d}$-valued locally bounded predictable process,
\item $(Q^j)_{j \in \{1, \ldots, p\}}$ is an adapted c\`agl\`ad process taking values in $S_d$,
\item $\mu^X$ is the random measure associated with the jumps of $X$, whose compensator is given by $F_t(d\xi)dt$, where $F_t(d\xi)=F_t(\omega,d\xi)$ is for each $(\omega,t)$ a measure on $S_d$ such that the process
\begin{align}\label{eq:comproot}
\left(\int_{S_d} (\sqrt{X_{t-}+\xi}-\sqrt{X_{t-}}) F_t(d\xi)\right)_{t \geq 0}
\end{align} 
is locally bounded.
\end{itemize}
Moreover, both processes $X_t$ and $X_{t-}$ take their values in $S_d^{++}$, the set of all (strictly) positive semidefinite $d \times d$ matrices.
Furthermore, the drift process $b^Y$ of $Y$ is additionally assumed to be adapted and c\`agl\`ad.
\end{assumption*}

\begin{remark}
\begin{enumerate}
\item Assumption $(H1)$ corresponds essentially to~\cite[Assumption 4.4.3 (or $(K)$)]{JP:12} or~\cite[Assumption $(H1)$]{BNGJPS:06}, respectively. The main difference is that we require $X$ to be an It\^o-semimartingale, whereas in the above references this condition is explicitly stated for $\sqrt{X}$. Since we assume additionally that $X_t$ and $X_{t-}$ take values in $S_d^{++}$, $\sqrt{X}$ is again an It\^o-semimartingale. Local boundedness then also holds for the drift and for the compensator of the jumps of
$\sqrt{X}$, where the latter is a consequence of condition~\eqref{eq:comproot}.
The motivation to state Assumption $(H1)$ in terms of $X$ stems again from applications to $S_d^+$-valued affine processes, where
the characteristics of $\sqrt{X}$ would have a much more complicated form than the simple affine dependence on $X$. 
\item Also in view of affine processes we prefer the formulation in terms of a Brownian motion $B$, correlated with $Z$ instead of decomposing $B$ into $Z$ and another independent Brownian motion. 
\item Concerning the jump part we implicitly assume that the jumps of $X$ are of finite variation, whence we can avoid the introduction of a truncation function. This is not restrictive in our case, since in the central limit theorem below we shall assume finite jump activity.
\end{enumerate} 
\end{remark} 
 
\subsection{The Role of Pathwise Covariance Estimation in Robust Calibration}\label{sec:4}

As outlined in Section~\ref{sec:1}, one goal of robust calibration is to estimate quantities which do not change under equivalent measures, such as the volatility of volatility, from time series observations. This necessitates to first reconstruct the path of the instantaneous covariance in a robust way and then to infer the second order quantities from this estimate.
More precisely, the time series estimation part of robust calibration consists in conducting the following steps:
\begin{enumerate}
\item[(1)]\label{st:1} 
the first step is to reconstruct (estimate) \emph{non-parametrically} from \emph{discrete observations} along an equidistant time grid $ \Lambda_1 = \{ t^n_0 < \ldots < t^n_{\lfloor nT \rfloor}=T \} $ with step width $ \frac{1}{n} $ of the log-price process $ Y $ a discrete set of points on the trajectory of the instantaneous covariance process $X$ along a coarser, equidistant time grid $ \Lambda_2 \subset \Lambda_1 $ with an appropriately chosen step width $ \frac{1}{N} $:
\[
\textrm{from} \quad{(Y_t)}_{t \in \Lambda_1} \quad\textrm{to}\quad {(\widehat{X}_t)_{t \in \Lambda_2}} \, .
\]
\item[(2)]\label{st:2}  
using then the reconstructed path $\widehat{X}$ along the coarser grid $ \Lambda_2 $ allows in principle to repeat the first step, or at least to estimate integrated quantities of that discretely given trajectory, for instance to compute an estimator for the integrated covariance of $X$:
\[
\textrm{from} \quad {(\widehat{X}_t)_{t \in \Lambda_2}} \quad  \textrm{to}\quad \widehat{\int_0^T q_s ds},
\]
where $q$ is defined by $q_{iji'j'}=\sum_{k,l} Q_{ij}^k Q_{i'j'}^l$.
\item[(3)]\label{st:3} under some parametric specification of $X$, e.g., being an affine $S_d^+$-valued process, certain parameters associated to the covariance of $X$ and the correlation between $X$ and $Y$ can then be estimated from the previously defined estimators.
\end{enumerate}

In order to perform the first step (1), that is, the non-parametric pathwise covariance estimation, we rely on (a modification of) the Fourier method introduced in~\cite{MM:02, MM:09}. In order to describe its main idea, let us first introduce some notation: for an $L^1[0,T]$ function $f$ we denote its Fourier coefficients for $k \in \mathbb{Z}$ by
\[
\mathcal{F}(f)(k)= \frac{1}{T} \int_0^T f(t)e^{-\im\frac{2\pi}{T}kt}dt.
\]
The Fourier method is now best described by the following steps. Notice, however, that we could perform these steps for any orthonormal system in $ L^2([0,T]) $.

\begin{enumerate}
\item[(1a)]\label{st:1a}  Recover from a discrete observation of $Y$ an estimator for the Fourier coefficients of the components of the path $t \to \rho(X_t(\omega))$ for some continuous invertible function $\rho: S_d \to S_d$. In other words, find an estimator for 
\[
\mathcal{F}(\rho(X))(k)= \frac{1}{T} \int_0^T \rho(X_t)e^{-\im\frac{2\pi}{T}kt}dt \, .
\]
\item[(1b)] Use Fourier-F\'ejer inversion to reconstruct the path of $t \to \rho(X_t)$. In fact, by F\'ejer's theorem
\[
 \widehat{\rho(X)}^N_t:=\sum_{k=-N}^{N} \left(1-\frac{|k|}{N}\right) \mathcal{F}(\rho(X))(k)e^{\im \frac{2\pi}{T} kt}
\]
converges uniformly (and in $L^2$) to $t \mapsto \rho(X_{t})$ on $[0,T]$ if $t \to X_{t}$ is continuous. If $X$ has c\`adl\`ag paths, then the limit is given by $\frac{\rho(X_{t})+\rho(X_{t-})}{2}$. Due to central limit theorems on the fine grid $ \Lambda_1 $ we make errors in the reconstruction of $ \rho(X) $ of size
\[
\sum_{k=-N}^{N} \left(1-\frac{|k|}{N}\right) \mathcal{E}_k e^{\im \frac{2\pi}{T} kt} \, ,
\]
where $ \mathcal{E}_k $ is a sequence of error random variables, which are approximately conditionally Gaussian with variance of order $ \mathcal{O} \big( n^{-1} \big) $. Hence it does not make sense to use \emph{all} Fourier coefficients from $ - \lfloor nT \rfloor $ to $ \lfloor nT \rfloor $, but there will appear a subtle relationship between the sum of the errors, the size of $ N $ with respect to $ n $ and the rate of the  central limit theorem for the reconstruction. 
\item[(1c)] Invert the function $\rho$ to obtain an estimator $\widehat{X}$ of the realized path of $X$.
\end{enumerate}

For the second step (2), we can rely on existing estimators for the realized integrated covariance, into which we plug the estimator $\widehat{X}$ obtained in the first step. For similar approaches to estimate such second order quantities, however based on the classical local realized variance estimator, compare~\cite{V:12,MW:12}.

In the third step (3) we focus on particular parametric specifications of $X$. One particularly tractable class is the class of affine models. In this case the quadratic variation of $X$ satisfies
\[
\langle X^c_{ii}, X^c_{jj}\rangle_T =4\alpha_{ij}\int_0^T X_{s,ij} ds, \quad \alpha \in S_d^+,
\]
such that knowing an estimator for $\langle X^c_{ii}, X^c_{jj}\rangle_T$ and an estimator for
\[
\langle Y^c_{i}, Y^c_{j}\rangle_T =\int_0^T X_{s,ij} ds,
\]
namely $\int_0^T \widehat{X}_{s,ij} ds$ allows to estimate the parameter $\alpha_{ij}$.

\section{Main results}\label{sec:main_theorem}

In order to state the main results and to introduce the estimators, let us make some assumptions on the observations of the log-price process $Y$. 
Throughout let $T >0$ be fixed and suppose that the time grids of observations for all components of $Y$ in $[0,T]$ are equal and equidistant, i.e.,
\[
 t^n_m=\frac{m}{n}, \quad m=0, \ldots, \left\lfloor nT \right\rfloor.
\]
The increments of a process $Z$ with respect to the above time grid are denoted by $\Delta^n_m Z=Z_{t^n_{m}}-Z_{t^n_{m-1}}$.

\begin{remark}
If grids are non-equidistant and non-equal for different coordinates it might be wise to use estimators, whose input are more continuous quantities than increments, e.g., Fourier coefficients. This is outlined for instance in \cite{MM:02, MM:09}. In any case our method will provide as a result continuous path functionals such as Fourier coefficients after the first estimation procedure.
\end{remark}

\subsection{Consistency and a Central Limit Theorem for Estimators of the Fourier Coefficients}

In this section we focus on step (1a), i.e., on how the Fourier coefficients of $t \mapsto \rho(X_t)$ can be estimated from discrete observations of $Y$. Realizing that the only difference with respect to estimators for integrated (functions of the) realized covariance are the terms $e^{-\im\frac{2\pi}{T}kt}$  in the integral for the Fourier coefficients, we can make use of (Fourier basis modified) jump robust estimators like
\begin{itemize}
\item the power variation estimators considered by Barndorff-Nielsen et al.~\cite{BNGJPS:06},
\item estimators for the realized Laplace transform of volatility introduced by Tauchen and Todorov~\cite{TT:11} and 
\item other jump robust specifications, as for example considered in~\cite[Theorem 5.3.5]{JP:12}
\end{itemize}
The estimators for the Fourier coefficients that we consider are of the form
\begin{align}\label{eq:Fouriercofest}
 V(Y,g,k)_T^n=\frac{1}{n}\sum_{m=1}^{ \left\lfloor nT \right\rfloor} e^{-\im\frac{2\pi}{T} kt_{m-1}^n}g(\sqrt{n}\Delta_m^n Y),
\end{align}
for some function $g:\mathbb{R}^d \to S_d$ and we write 
\[
 V(Y,g)_T^{n,N}:=(V(Y,g,-N)_T^n, \ldots, V(Y,g,0)_T^n, \ldots, V(Y,g,N)_T^n)^{\top}.
\]
Note that the $0$-th Fourier coefficient $V(Y,g,0)_T^n$ corresponds to estimators for integrated (functions of the) covariance, as the power variation estimators, but also the realized Laplace transform estimator. Indeed in these cases the function $g$ is given by
\[
g: \mathbb{R}^d \to S_d, \quad (x_1, \ldots, x_d)^{\top} \mapsto (|x_i|^r|x_j|^s)_{i,j \in \{1, \ldots, d\}}, \quad r+s >0
\]
and
\[ 
 g: \mathbb{R}^d \to S_d, \quad (x_1, \ldots, x_d)^{\top} \mapsto (\cos(x_i+x_j))_{i,j \in \{1, \ldots, d\}}
\]
respectively.

\begin{remark}
Let us remark that the estimators for the Fourier coefficients of $t \mapsto X_t$ introduced in~\cite{MM:02, MM:09} are defined via the so-called Bohr convolution of the Fourier coefficients of $t \mapsto Y_t$. 
In the case of the $0$-th Fourier coefficient estimator, i.e., in the case of integrated covariance, this specification leads in particular to robustness towards microstructure noise due the presence of Dirichlet-kernel weighted auto-covariances. For jump processes it is, however, not obvious how to generalize the Fourier estimators based on the Bohr convolution. One possibility is to introduce a process $R$ via
\[
dR_t=\left(\sum_{m=1}^{ \left\lfloor nT \right\rfloor} g(\sqrt{n}\Delta_m^n Y)1_{[t^n_{m-1}, t^n_m)}(t) \right)dW_t,
\]
where $W$ is a Brownian motion independent of $Y$. Defining  estimators for the Fourier coefficients of $t \mapsto X_t$ via the Bohr convolution of the Fourier coefficients of $R$,  yields a similar expression as~\cite[Equation (24)]{MM:09} involving the Dirichlet kernel. In contrast to~\eqref{eq:Fouriercofest}, such an estimator would then enjoy similar properties as the one proposed in~\cite{MM:09}, in particular with respect to microstructure noise. The analysis of these estimators is beyond the scope of the paper.
\end{remark}

Our first aim is to study asymptotic properties of  $V(Y,g)_T^{n,N}$, for which we rely to a large extent on the results of~\cite{BNGJPS:06}, ~\cite{J:08} and~\cite{JP:12}. The following assumptions on the function $g$, needed to establish consistency and a central limit theorem, are also taken from~\cite{BNGJPS:06}:

\begin{assumption*}[J]
The function $g$ is continuous with at most polynomial growth.
\end{assumption*}

\begin{assumption*}[K]
The function $g$ is even and continuously differentiable with partial derivatives having at most polynomial growth.
\end{assumption*}

\begin{assumption*}[K']
The function $g$ is even, with at most polynomial growth and $C^1$ outside a subset $B$ of $\mathbb{R}^d$ which is a finite union of affine hyperplanes. With $d(x,B)$ denoting the distance between $x \in \mathbb{R}^d$ and $B$, we have for some $w \in (0,1]$ and $p \geq 0$
\begin{align*}
&x \in B^c \Rightarrow \| \nabla g(x)\| \leq C (1+\|x\|^p)\left(1+\frac{1}{d(x,B)^{1-w}}\right)\\
&x \in B^c, \; \|y\|\leq (1 \wedge \frac{d(x,B)}{2}) \\
&\quad \Rightarrow \|\nabla f(x+y)-\nabla f(x)\|\leq C\|y\|(1 +\|x\|^p+\|y\|^p)\left(1+\frac{1}{d(x,B)^{2-w}}\right).
\end{align*}
\end{assumption*}

\begin{remark}
The conditions of Assumption $(K')$ are especially designed to accommodate the functions
\[
(x_1, \ldots, x_d)^{\top} \mapsto (|x_i|^r|x_j|^{s})_{i,j \in \{1, \ldots, d\}}
\] 
for $r+s <1$, which correspond to the jump robust power variation estimators.
\end{remark}

In the case when $Y$ is a pure diffusion process the results of~\cite[Theorem 2.1. and Theorem 2.3]{BNGJPS:06} carry directly over to the Fourier basis modified statistics $V(Y,g,k)_T^n$. In the case of jumps, the respective assertions of~\cite[Theorem 3.4.1, Theorem 5.3.5 and Theorem 5.3.6]{JP:12} can also be directly transfered to $V(Y,g,k)_T^n$.
A sufficient condition which allows to incorporate jumps and which is also satisfied by the assumptions of the cited theorems, relates the function $g$ with the jump activity of (a localized version of) $Y$ and is stated in Assumption $(L(\eta))$ below.
Let us denote by $D^Y$ the diffusion part of $Y$ with respect to some truncation function $\chi'$, i.e.,
\begin{align}\label{eq:diffpart}
D^Y_t(\chi')=y+\int_0^t \left(b^Y_s +\int (\chi'(\xi)-\chi(\xi))K_s(d\xi)\right)ds + \int_0^t\sqrt{X_{s-}} dZ_s.
\end{align}
Then we shall require that the $L^1$-norm of $g(\sqrt{n}\Delta_m^n Y(p))- g(\sqrt{n}\Delta_m^n D^{Y(p)}(\chi'))$ goes sufficiently fast to $0$ uniformly in $m$, where $Y(p)$ denotes a localized version of $Y$.

\begin{assumption*}[$L(\eta)$]
Let $\chi'$ be a truncation function such that the modified drift of $Y$
\begin{align}\label{eq:moddrift}
b^Y(\chi')=b^Y +\int(\chi'(\xi)-\chi(\xi))K(d\xi)
\end{align}
is c\`agl\`ad. Moreover, suppose that there exists an increasing sequence of stopping times $(\tau_p)$ with $\lim_p \tau_p =\infty$ a.s. and processes $Y(p)$ such that
for $t< \tau_p$
\begin{align}\label{eq:Yloc}
Y(p)_t=Y_t \textrm{ a.s and } D_t^Y(\chi')=D_t^{Y(p)}(\chi') \textrm{ a.s.}
\end{align}
For $\eta \geq 0$, we then have for all $p$
\[
\lim_{n \to \infty} \sup_{1\leq m\leq \lfloor nT \rfloor}  n^{\eta}\mathbb{E}\left[\left\|g(\sqrt{n}\Delta_m^n Y(p))- g(\sqrt{n}\Delta_m^n D^{Y(p)}(\chi'))\right\|\right]=0.
\]
\end{assumption*}

In Corollary~\ref{cor:jumprobust} and Section~\ref{sec:jumpspec} below, we shall give precise examples of $g$ for which this condition is satisfied.
 
For the formulation of our first result we need some further notation: let $f: [0,T] \to S_d$ be some $L^1([0,T])$ function. Then we denote the $(2N+1)d \times d$ dimensional ``vector'' of Fourier coefficients by
\[
\mathcal{F}^N(f)=(\mathcal{F}(f)(-N),\ldots, \mathcal{F}(f)(0), \ldots, \mathcal{F}(f)(N))^{\top} \, .
\]
Moreover, for a function $h: \mathbb{R}^d \to \mathbb{R}^m$ and a $d$-dimensional normally distributed random variable $U$ with mean $0$ and covariance $X$, the first moment of $h(U)$ is denoted by $\rho_h(X)$, i.e.,
\[
 \rho_h(X)=\mathbb{E}\left[h(U)\right],\quad U\sim \mathcal{N}(0,X).
\]
By $\mathcal{F}(\rho_h(X))(k)$ we then mean
\[
\mathcal{F}(\rho_h(X))(k)= \frac{1}{T}\int_0^T \rho_h(X_t)e^{-\im\frac{2\pi}{T}kt}dt.
\]

\begin{theorem} \phantomsection\label{th:Fouriercofestconv}
\begin{enumerate}
\item Under Assumptions $(H)$, $(J)$ and $(L(0))$, we have
\[
 V(Y,g)_T^{n,N} \stackrel{\mathbb{P}}{\to} T\mathcal{F}^N(\rho_g(X)).
\]
\item \label{th:Fouriercofclt} Under the assumption $(H1)$ and $(K)$ or $(K')$ and $(L(\frac{1}{2}))$, the $\mathbb{C}^{(2N+1)d\times d}$-valued random variable
\[
\sqrt{n} \left(V(Y,g)_T^{n,N}- T\mathcal{F}^N(\rho_g(X))\right)
\]
converges for $n \to \infty$ stably in law to an  $\mathcal{F}$-conditional Gaussian random variable defined on an extension of the original probability space with mean $0$ 
and covariance 
\begin{align*}
C_{iji'j'}^{kk'}&:=\int_0^T \left(\rho_{g_{ij}g_{{i'}{j'}}}(X_s)-\rho_{g_{ij}}(X_s)\rho_{g_{i'j'}}(X_s)\right)e^{-\im\frac{2\pi}{T}(k-k')s} ds, \\
\end{align*}
where $i,j,i',j' \in \{1, \ldots, d\}$ and $k, k' \in \{-N, \ldots, N\}$.
\end{enumerate}
\end{theorem}

\begin{remark}\phantomsection\label{rem:Fouriercoeff}
 \begin{enumerate}
\item \label{rem:it1} Stable convergence in law for a sequence of random variables $(U_n)$ to a limit $U$ 
(defined on an extension of $(\Omega, \mathcal{F}, \mathbb{P})$) means that, for any bounded continuous function $f$ and any bounded $\mathcal{F}$-measurable random variable $V$, we have
\[
\lim_{n \to \infty}\mathbb{E}\left[Vf(U_n)\right]=\mathbb{E}\left[Vf(U)\right].
\]
\item The above convergence results do not only hold for $T$ fixed, but we have 
\[
 V(Y,g)_T^{n,N} \stackrel{\mathbb{P}}{\to} T\mathcal{F}^N(\rho_g(X))
\]
locally uniformly in $T$ and also stable convergence process-wise.\footnote{Here, $\mathcal{F}(f)(k)$ is defined for variable $T$.} The latter means that
\[
\sqrt{n} \left(V(Y,g)_T^{n,N}- T\mathcal{F}^N(\rho_g(X))\right)
\]
converges stably in law to a process $U(g,N)$ given componentwise by 
 \begin{align}\label{eq:UgN}
  U(g,N)_{ij,T}^k=\sum_{k'=1}^{2(2N+1)}\sum_{i',j'=1}^{d}\int_0^T \delta_{s,ij, i'j'}^{kk'} dW^{k'}_{s,i'j},
 \end{align}
where
\[
\sum_{r=1}^{2(2N+1)}\sum_{p,q=1}^{d}\int_0^T   \delta_{s,ij,pq}^{kr} \overline{\delta_{s,i'j',pq}^{k'r}}=C_{iji'j'}^{kk'}.
\]
Here, $W$ is a $2(2N+1)d \times d$-dimensional Brownian motion which is defined on an extension of the probability space 
 $(\Omega, \mathcal{F}, (\mathcal{F}_t)_{t \geq 0}, \mathbb{P})$ and is independent of the $\sigma$-field $\mathcal{F}$.
\item The above theorem has been proved in~\cite{BNGJPS:06} in a pure diffusion setting and $k=0$. Inclusion of jumps has been considered (in the one-dimensional case) in~\cite{BNSW:06} and~\cite{W:06} for $g=|x|^r$ and in~\cite{TT:11} for $g=\cos(x)$. 
More general functions (also for the case $k=0$) are treated in~\cite{J:08} and~\cite[Theorem 3.4.1, Theorem 5.3.5 and Theorem 5.3.6]{JP:12}.
\item 
In the examples  $g(x)=|x|^r$ and $g(x)=\cos(x)$, the function $\rho_g(x)$ corresponds to
\[
 \rho_{(x \mapsto |x|^r)}(x)=|x|^\frac{r}{2}\mathbb{E}\left[|U|^r\right], \, U \sim \mathcal{N}(0,1)
\]
and 
\[
 \rho_{(x \mapsto \cos(x))}(x)=e^{-\frac{1}{2}x} \, ,
\]
respectively.
\end{enumerate}
\end{remark}

In the following corollary we specify classes of functions $g$ and conditions on the jumps such that condition $(L(0))$ or $(L(\frac{1}{2}))$, respectively, is satisfied and the estimator $V(Y,g,k)_T^n$ given in~\eqref{eq:Fouriercofest} is robust to jumps. These conditions are in line with the respective assumptions in~\cite[Theorem 3.4.1 (a), Theorem 5.3.5 ($\gamma$) and Theorem 5.3.6, Equation 5.3.11]{JP:12}.

\begin{corollary}\phantomsection\label{cor:jumprobust}
\begin{enumerate}
\item\label{cor:jumprobust1}
Let $g$ be continuous with $g(x)=o(\|x^2\|)$ as $\|x \| \to \infty$. Then under assumption $(H)$ we have 
\[
 V(Y,g)_T^{n,N} \stackrel{\mathbb{P}}{\to} T\mathcal{F}^N(\rho_g(X)).
\]
\item\label{cor:jumprobust2}
Suppose that $g$ satisfies for some $q\geq 0$ and some $0< r \leq r'<1$
\begin{align*}
\|g(x)-g(y)\|\leq C (1+\|y\|^q)(\|x-y\|^{r}+\|x-y\|^{r'}).
\end{align*}
and assume that $\mathbb{E}\left[\int_{\|\xi\|> 1}\|\xi\| K_{t}(d\xi)\right]< \infty $ holds true.
Moreover, let $\beta \in [0,1)$ and assume that for all $t \in [0,T]$
 \begin{align*}
 \mathbb{E}\left[\int_{\|\xi\|\leq 1}\|\xi\|^{\beta} K_{t}(d\xi)\right]& < \infty,
 \end{align*}
and that 
\[
b^Y(0)=b^Y-\int \chi(\xi)K(d\xi)
\]
is c\`agl\`ad. Then under the assumptions $(H1)$ and $(K)$ or $(K')$ and $\frac{\beta}{2-\beta} < r <1$, 
the central limit theorem of Theorem~\ref{th:Fouriercofestconv}~\ref{th:Fouriercofclt} holds true.
\end{enumerate}
\end{corollary}

\begin{remark}
\begin{enumerate}
\item 
The specifications 
\begin{align}\label{eq:powervariationjump}
g: \mathbb{R}^d \to S_d, \quad (x_1, \ldots, x_d)^{\top} \mapsto (|x_ix_j|^{\frac{r}{2}})_{i,j \in \{1, \ldots, d\}}
\end{align}
for $r <1$ and
\[ 
 g: \mathbb{R}^d \to S_d, \quad (x_1, \ldots, x_d)^{\top} \mapsto (\cos(x_i+x_j))_{i,j \in \{1, \ldots, d\}}
\]
are covered by these conditions. In the case of~\eqref{eq:powervariationjump}, the above corollary recovers~\cite[Theorem 1~(iii)]{BNSW:06}, which has been proved for one dimensional jump diffusions where the jumps are described by a L\'evy process.
For functions satisfying~\eqref{eq:growth} a similar statement is proved in~\cite[Theorem 5.3.5 ($\gamma$) and Theorem 5.3.6]{JP:12}, however, under slightly different conditions on the jump measures (in particular, supposing the Gringelions representation of $Y$).
\item Another function which satisfies for example the above requirements and for which $\rho_g$ is invertible and easily computable is
\[
g: \mathbb{R}^d \to S_d, \quad (x_1, \ldots, x_d)^{\top} \mapsto \left(e^{-\frac{\langle x, A_{ij} x\rangle}{2}}\right)_{ij\in \{1, \ldots, d\}},
\]  
where  $A_{ij}=e_ie_i^{\top}+e_je_j^{\top}+e_ie_j^{\top}+e_je_i^{\top}$ and $e_i$ denotes the canonical basis vector.
The function $\rho_g$ is given by  
\[
\rho_{g_{ij}}(X)=\frac{1}{\sqrt{X_{ii}+2X_{ij}+X_{jj}+1}}.
\]
\end{enumerate}
\end{remark}

\subsection{Asymptotic Properties of Estimators for (Functions of) the Instantaneous Covariance Process}

We now focus on step (1b) and (1c), that is, we are interested in establishing consistency and a central limit theorem for an estimator of $\rho_{g}(X_t)$ 
and $X_t$ respectively. The estimator for  $\rho_{g}(X_t)$ is defined via Fourier-F\'ejer inversion using the above estimators for the Fourier coefficients:
\begin{align}\label{eq:estimatorrhoX}
\widehat{\rho_{g}(X)}^{n,N}_{t}=\frac{1}{T}\sum_{k=-N}^{N} \left(1-\frac{|k|}{N}\right)e^{\im \frac{2\pi}{T} kt}V(Y,g,k)_T^{n}.
\end{align}
Once we have obtained a consistency and a central limit theorem for this estimator, we can translate these results 
to an estimator for $X_t$, which we define via
\begin{align}\label{eq:estimatorX}
\widehat{X}_t^{n,N}:=\rho^{-1}_g\left(\widehat{\rho_{g}(X)}^{n,N}_{t}\right)
\end{align}
provided that $\rho_{g}(x): S_d \times S_d,\, x \mapsto \rho_g(x)$ is invertible. 

\subsubsection{Consistency}

Let us start with the consistency statements.

\begin{theorem}\label{th:consistency}
Let $\gamma > 1$ and suppose that $\lim \frac{n}{N^{\gamma}}=K$ for some constant $K >0$. Under the assumptions $(H)$, $(J)$ and $(L(0))$
we have for every $t \in [0,T]$
\[
 \widehat{\rho_{g}(X)}^{n,N}_{t} \stackrel{\mathbb{P}}{\to} \frac{\rho_g(X_{t-})+\rho_g(X_t)}{2}
\]
as $n, N \to \infty$. If $X$ has no fixed time of discontinuity, then 
\[
 \widehat{\rho_{g}(X)}^{n,N}_{t} \stackrel{\mathbb{P}}{\to}\rho_g(X_t).
\]
\end{theorem}

The following corollary states explicit conditions on $g$ and the jumps of $Y$ such that $(L(0))$ is satisfied and relies on Proposition~\ref{prop:jumprobust}~\ref{prop:jumprobustLLN} below.

\begin{corollary}\label{cor:consistency}
Let $g$ be continuous with $g(x)=o(\|x^2\|)$ as $\|x \| \to \infty$. Let $\gamma > 1$ and suppose that $\lim \frac{n}{N^{\gamma}}=K$ for some constant $K>0$. Then under assumption $(H)$ we have for every $t \in [0,T]$
\[
 \widehat{\rho_{g}(X)}^{n,N}_{t} \stackrel{\mathbb{P}}{\to} \frac{\rho_g(X_{t-})+\rho_g(X_t)}{2}
\]
as $n, N \to \infty$.
 If $X$ has no fixed time of discontinuity, then 
\[
 \widehat{\rho_{g}(X)}^{n,N}_{t} \stackrel{\mathbb{P}}{\to}\rho_g(X_t).
\]
\end{corollary}

\begin{proof}
The proof is a consequence of Theorem~\ref{th:consistency} and Proposition~\ref{prop:jumprobust}~\ref{prop:jumprobustLLN} below.
\end{proof}

We can now transfer the consistency result to the instantaneous covariance estimator~\eqref{eq:estimatorX}.

\begin{corollary}\label{cor:consistentX}
Let $g$ be such that $\rho_{g}(x): S_d \times S_d,\, x \mapsto \rho_g(x)$ has a continuous inverse.
Then under the assumptions of Theorem~\ref{th:consistency} or Corollary~\ref{cor:consistency}
we have for every $t \in [0,T]$
\[
 \widehat{X}^{n,N}_{t} \stackrel{\mathbb{P}}{\to} \rho_g^{-1}\left(\frac{\rho_g(X_{t-})+\rho_g(X_t)}{2}\right)
\]
as $n, N \to \infty$.  If $X$ has no fixed time of discontinuity, then 
\[
 \widehat{X}^{n,N}_{t} \stackrel{\mathbb{P}}{\to} X_t.
\]
\end{corollary}

\begin{proof}
 This corollary is simply a consequence of the continuous mapping theorem.
\end{proof}

\subsubsection{Central Limit Theorem}

In this section we formulate an asymptotic normality result for the estimators 
\[
\widehat{\rho_{g}(X)}^{n,N}_{t} \quad \textrm{ and } \quad \widehat{X}^{n,N}_{t}.
\]
For this we assume additionally that the covariance process $X$ has no fixed time of discontinuity, that its jumps are of finite activity and that the trajectories between two jumps are almost surely H\"older continuous with some exponent $\delta$.

\begin{theorem}\label{th:Fourierestconv}
Assume that $X$ has no fixed time of discontinuity and that its jumps are of finite activity. Suppose that the trajectories of $X$ between two jumps are almost surely H\"older continuous with some exponent $\delta$.
Let $1 <\gamma < 2\delta +1$ and suppose that $\lim \frac{n}{N^{\gamma}}=K$ for some constant $K>0$.  Then under ($H1$) and $(K)$ or $(K')$ and $(L(\eta))$ with $\eta \geq \frac{\gamma-1}{2\gamma}$, the random variable
\begin{align}~\label{eq:limit}
\sqrt{\frac{n T}{N}}\left(\widehat{\rho_{g}(X)}^{n,N}_{t}-\rho_{g}(X_t)\right)
\end{align}
converges for each $t \in [0,T]$ as $n,N \to \infty$ stably in law to an $\mathcal{F}$-conditional Gaussian random variable defined on an extension of the original probability space with mean $0$ and covariance function given by 
\begin{align*}
V^{iji'j'}_t:&=\frac{2}{3}(\rho_{g_{ij}g_{i'j'}}(X_t)-\rho_{g_{ij}}(X_t)\rho_{g_{i'j'}}(X_t)).
\end{align*}
\end{theorem}

Similarly as above for the consistency statement we now translate the central limit theorem to $\widehat{X}_t^{n,N}$ defined in~\eqref{eq:estimatorX}.

\begin{corollary}\label{cor:cltestimatorX}
Let $g$ be such that $\rho_{g}(x): S_d \times S_d,\, x \mapsto \rho_g(x)$ has a differentiable inverse.
Then under the assumptions of Theorem~\ref{th:Fourierestconv}
\begin{align*}
\sqrt{\frac{n T}{N}}\left(\widehat{X}_t^{n,N}-X_t\right)
\end{align*}
converges as $n, N \to \infty$ for each $t \in [0,T]$ stably in law to a $\mathcal{F}$-conditional Gaussian random variable given by
\[
M_t=(\nabla \rho_{g}(X_t))^{-1}N_t,
\]
where $N_t$ denotes the limit of~\eqref{eq:limit}.
In particular the covariance of $M_t$ is given by
\[
U_t^{iji'j'}=\sum_{k,l, k',l'} (\nabla \rho_{g}(X_t))_{ijkl}^{-1}(\nabla \rho_{g}(X_t))_{i'j'k'l'}^{-1} V_t^{klk'l'}.
\]
\end{corollary}
\begin{proof}
 This assertion follows from the concept of stable convergence and is known as generalized $\Delta$-method (see~\cite[Theorem 1.10]{P:06}).
\end{proof}

\begin{remark}\phantomsection\label{rm:CLTremarks}
\begin{enumerate}
\item In the above theorem the assumption that the trajectories of $X$ are $\delta$-H\"older continuous between two jumps is crucial for the relation between $n$ and $N$ and thus for the speed of convergence.
Typically H\"older continuity of paths can be verified by using Kolmogorov's criterium, which states that the H\"older exponent $\delta$ satisfies $\delta < \frac{\varepsilon}{\alpha}$ if
\[
\mathbb{E}\left[\| X_t-X_s\|^{\alpha}\right] \leq C |t-s|^{1+\varepsilon}.
\]
In our case, when we assume no jumps at all, we have for $\alpha >2$
\begin{align*}
& \mathbb{E}\left[\left\| \int_s^t b^X_u du +\sum_{j=1}^p \int_s^t Q_{u}^j dB_{u,j}\right\|^{\alpha}\right] \\
\leq & C\left(\mathbb{E}\left[\left\| \int_s^t b^X_u du \right\|^{\alpha}\right] +\sum_{j=1}^p \mathbb{E}\left[\left\| \int_s^t Q_{u}^j dB_{u,j}\right\|^{\alpha}\right]\right)\\ 
\leq & C\left(\mathbb{E}\left[\int_s^t |t-s|^{\alpha-1}\|b^X_u\|^{\alpha} du \right] +\sum_{j=1}^p \mathbb{E}\left[\left\| \int_s^t (Q_{u}^j)^2 du\right\|^{\frac{\alpha}{2}}\right]\right)\\
\leq & C|t-s|^{\alpha}+C\sum_{j=1}^p \mathbb{E}\left[ \int_s^t |t-s|^{\frac{\alpha}{2}-1} \|Q_{u}^j\|^{\alpha} du\right]\\
\leq & C|t-s|^{\alpha}+C|t-s|^{\frac{\alpha}{2}}\\
\leq & C|t-s|^{\frac{\alpha}{2}}
\end{align*}
provided that $\mathbb{E}\left[\|b^X_u\|^{\alpha}\right]< \infty $ and $\mathbb{E}\left[  \|Q_{u}^j\|^{\alpha} \right]< \infty$.
For the last inequality we used the fact that we are working on $[0,T]$.
If these moments exist for all $\alpha$ then we have H\"older continuity with $\delta < \frac{1}{2}$.
This is for example satisfied for affine diffusion processes on $S_d^+$.
\item The convergence rate in the above central limit theorem is of order $n^{\frac{\gamma-1}{2\gamma}}$. If we have H\"older continuity (between two jumps) for all $\delta < \frac{1}{2}$, then $\gamma \in (1,2)$.  The higher $\gamma$ the better the convergence rate and it lies between $(0, \frac{1}{4})$ in this case.
If the paths of the covariance are even $k$-times differentiable for some $k \in \mathbb{N}$ and the $k^{th}$ derivative is H\"older continuous with exponent $\delta$, then $\gamma$ can be chosen to lie in $(1,2(k+\delta)+1)$. In the case of a constant covariance process, the convergence rate is thus $\sqrt{n}$.
\item Condition $(L(\eta))$ with $\eta \geq \frac{\gamma-1}{2\gamma}$ is satisfied, if the assumptions of Proposition~\ref{prop:jumprobust}~\ref{prop:jumprobustCLT} below
hold with $r>\frac{\gamma-1}{\gamma} \frac{\beta}{2-\beta}$ and $\beta < 1 +\frac{1}{2\gamma-1}$.
\item \label{rm:Fourierestconv}
As already mentioned in the introduction the asymptotic variance constant of the presented Fourier-F\'ejer estimator is smaller that the one of the classical local realized variance estimator, while both estimators have the same rate of convergence. Notice here also the asymptotic equivalence of spot variance regression with well-understood Gaussian shift models in Le Cam's sense, see, e.g., \cite{rei:08}. Therefore the following analogy with well-known shrinkage estimators does make sense:

For simplicity, let us consider the one-dimensional case with $g(y)=y^2$. Then the variance of the Fourier-F\'ejer estimator equals
\begin{align}\label{eq:varFourier}
\frac{4}{3}X_t^2
\end{align}
under the assumptions of the above theorem. In comparison, consider the classical (non-truncated) local realized variance estimator given by
\[
\widehat{\widehat{X}}^{n,N}_t=\sum_{j=1}^{N} \varepsilon^N_j(t) \sum_{m=1}^{ \left\lfloor nT \right\rfloor}  (\Delta_m^n Y)^2 \varepsilon^N_j(t^n_{m-1})
\]
with 
\[
\varepsilon^N_j(t)=1_{[t^N_{j-1},t^N_j]}(t)\frac{1}{\sqrt{t^N_{j}-t^N_{j-1}}} \quad \textrm{and} \quad  t^N_{j}=\frac{jT}{N},\, j=1, \ldots, N.
\]
Similar as in the above theorem, suppose $1 <\gamma < 2$ and $\lim \frac{n}{N^{\gamma}}=K$ for some constant $K>0$. 
Then, according to~\cite[Theorem 13.3.3 b)]{JP:12}
\[
\sqrt{\frac{nT}{N}}\left(\widehat{\widehat{X}}^{n,N}_t -X_t\right)
\]
converges for each $t \in [0,T]$ as $n,N \to \infty$ stably in law to an $\mathcal{F}$-conditional Gaussian random variable with mean $0$ and  covariance function given by 
$
2 X_t^2
$
and is therefore $\frac{3}{2}$ times bigger than~\eqref{eq:varFourier}. Let us remark that in the notation of~\cite[Theorem 13.3.3 b)]{JP:12}, $k_n$ corresponds to $\frac{n T }{N}$, that is, the number of points in the interval $[t^N_{j-1},t^N_j]$, $\tau=\frac{\gamma-1}{\gamma} \in (0,\frac{1}{2})$ and $\beta'=K^{\frac{1}{\gamma}}T$
such that the above assertion concerning the classical estimator is implied by~\cite[Theorem 13.3.3 a) and b)]{JP:12} with $\beta=0$ since $\tau< \frac{1}{2}$.

A similar variance reduction phenomenon can be achieved by applying the James-Stein estimator to $\widehat{\widehat{X}}^{n,N}_t$ and considering the following shrinkage estimator:
\[
\left(1- \frac{(M-2)2X^2N}{n \sum_{i=1}^M\left(\widehat{\widehat{X}}^{n,N}_{t_i}-\frac{1}{M}\sum_k \widehat{\widehat{X}}^{n,N}_{t_k}\right)^2}\right)\left(\widehat{\widehat{X}}^{n,N}_t-\frac{1}{M}\sum_{k=1}^M \widehat{\widehat{X}}^{n,N}_{t_k}\right)+\frac{1}{M}\sum_{k=1}^M \widehat{\widehat{X}}^{n,N}_{t_k},
\]
where $M$ denotes the number of evaluation points of $\widehat{\widehat{X}}^{n,N}$. We consider here the estimation of spot volatility, which naturally comes with an (asymptotically normal) noise, in the realm of estimation of drift in a noisy environment. In this setting the James-Stein methodology of shrinkage can improve estimator variances for the price of (small) biases, see, e.g., the infinite-dimensional recent work~\cite{prirev:08}. Notice that shrinkage towards the average of the spot-observations acts like a convolution with an almost ``flat-tailed'' kernel, which additionally behaves in the center like a parabola of the type $ 1 - x^2 $. This is related to two crucial properties of the Fourier-F\'ejer kernel, which again supports the Fourier approach.

The reduction of the estimator variance is confirmed by Figure~\ref{Fig:comp} below, which shows a comparison between the classical local realized variance estimator and the Fourier-F\'ejer estimator. In particular, the variance of the Fourier-F\'ejer estimator is comparable with the one of a James-Stein shrinkage variant of the classical estimator.
In our illustration example the underlying semimartingale $Y$ is a drifted Brownian motion with constant variance, that is,
\[
dY_t=b^Y_t dt+ \sqrt{X}dZ_t,
\]
where $b^Y$ denotes the drift, $Z$ a standard Brownian motion and $X$ the deterministic constant variance, which we aim to measure on a coarser grid given discrete observations of $Y$. 

\begin{figure}[ht]
\centering
\includegraphics[width=300pt,height= 200pt]{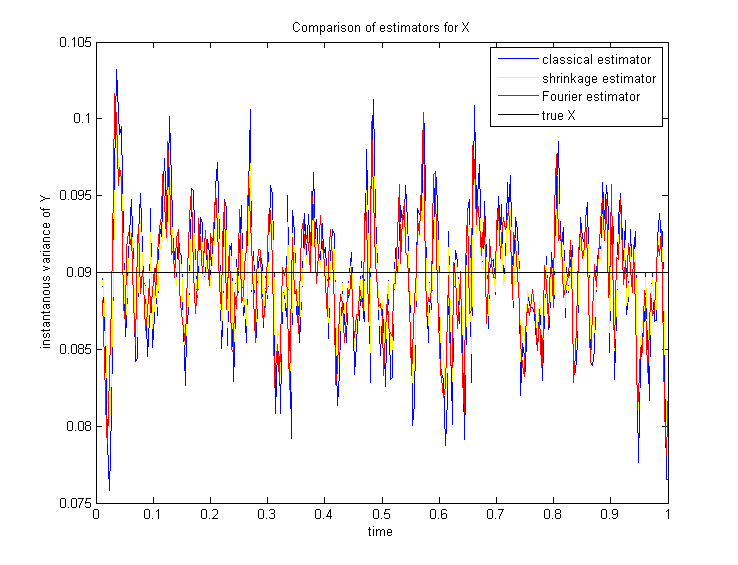}
\caption{Comparison of estimators}\label{Fig:comp}
\end{figure}

\end{enumerate}
\end{remark}

\subsection{Covariance of covariance estimation}
Having reconstructed the path of the instantaneous covariance, we can now proceed with step (2), that is, the estimation of functions of the integrated covariance. To this end we plug the reconstructed path of the instantaneous covariance process into
jump robust estimators of the form
\[
V(\widehat{X}^{n,N},f,0)_T^m:=\frac{1}{m}\sum_{p=1}^{ \left\lfloor mT \right\rfloor} 
 f(\sqrt{m}\Delta_p^m\widehat{X}^{n,N}),
\]
where $f: S_d \to \mathbb{R}^{d^2 \times d^2}$ satisfies certain properties which are specified in Theorem~\ref{th:covcov} below.

\begin{theorem}\phantomsection\label{th:covcov}
\begin{itemize}
\item Assume that the jumps of $X$ are of finite activity. Moreover, suppose that the paths of $X$ between two jumps are almost surely H\"older continuous with exponent $\delta$.
\item Let $1 <\gamma < 2\delta +1$ and suppose that $\lim \frac{n}{N^{\gamma}}=K$ for some constant $K>0$. 
\item Let the conditions $(H1)$ and $(K)$ or $(K')$ and $L(\eta)$ with $\eta \geq \frac{\gamma-1}{2\gamma}$ be in force 
and suppose that the instantaneous covariance process $q$ of $X$ defined by $q_{s,iji'j'}=\sum_{k,l} Q_{s,ij}^k Q_{s,i'j'}^l$ satisfies $(H1)$.
\item Let $g$ be such that  $x\mapsto(\nabla \rho_g(x))^{-1}$ exists and is continuous.  
\item Suppose that $f: S_d \to \mathbb{R}^{d^2 \times d^2}$ satisfies either $(K)$ or $(K')$ and is additionally 
globally $\alpha$-H\"older continuous for $\alpha \in (0,1)$.
\item Let $\iota < \frac{\gamma-1}{\gamma}\frac{\alpha}{1+\alpha}$ and assume that $\lim \frac{m}{n^{\iota}}=\widetilde{K}$ for some constant $\widetilde{K}>0$.
\end{itemize}
Then
\[
n^{\frac{\iota}{2}} \left( V(\widehat{X}^{n,N},f,0)_T^m-\int_0^T \rho_f(q_s) ds\right)
\]
converges as $n, N, m \to \infty$ stably in law to a $\mathcal{F}$-conditional Gaussian random variable defined on an extension of the original probability space with $0$ mean and covariance function given by
\[
F_{ijkli'j'k'l'}=\frac{1}{\widetilde{K}}\int_0^T\left(\rho_{f_{iji'j'}f_{klk'l'}}(q_s)+\rho_{f_{ijk'l'}}(q_s)\rho_{f_{i'j'kl}}(q_s)\right)ds.
\]
\end{theorem}

\section{Proofs of the Asymptotic Properties of the Fourier Coefficients Estimator} \label{sec:FC}

In this section we provide the proofs of Theorem~\ref{th:Fouriercofestconv} and Corollary~\ref{cor:jumprobust}.
Throughout $C$ always denotes some constant which can vary from line to line.
Moreover, let us introduce the following notation
\begin{align}
 \beta_{m}^n&:=\sqrt{n}\sqrt{X_{t^n_{m-1}}}\Delta_m^n Z,\label{eq:betamn}\\
  \rho_{m}^n(g)&:=\rho_g(X_{t^n_m}).\label{eq:rhomn}
\end{align}

\subsection{Proof of Theorem~\ref{th:Fouriercofestconv}}

\begin{proof}
We consider the one-dimensional case, i.e., $g: \mathbb{R} \to \mathbb{R}$, since the multi-dimensional case follows from it immediately in a methodological sense. Moreover, by the so called \emph{localization procedure} as described in~\cite[Section 3]{BNGJPS:06} or~\cite[Section 4.4.1]{JP:12}, we can strengthen assumption $(H)$ on $Y$ together with Condition~\eqref{eq:moddrift} to $(SH)$ (see~\cite[Section 3]{BNGJPS:06}), that is, $b^Y(\chi')$ 
defined in~\eqref{eq:moddrift}, and $X$ are supposed to be uniformly bounded by a constant. Similarly, assumption $(H1)$ on $X$ is strengthened to $(SH1)$, that is, $(SH)$ holds and the differential characteristics of the drift and the diffusion part of $\sqrt{X}$ and~\eqref{eq:comproot}, that is,
\[
\left(\int_{S_d} (\sqrt{X_{t-}+\xi}-\sqrt{X_{t-}}) \, F_t(d\xi)\right)_{t \geq 0}
\]
are bounded by a constant (compare~\cite[Assumption $(SH1)$]{BNGJPS:06} and~\cite[Assumption 4.4.7 (or (SK))]{JP:12}).

Let us denote by $\mathcal{B}^N$ the elements of the Fourier basis
\begin{align*}
 \mathcal{B}^N(t)&:=(e^{-\im\frac{2\pi}{T}(-N)t}, \ldots, 1,\ldots, e^{-\im\frac{2\pi}{T}(N)t})^{\top},
\end{align*}
and let us split
\begin{align*}
 &\sqrt{n} \left(V(Y,g)_T^{n,N}-T\mathcal{F}^N(\rho_{g}(X))\right)\\
 &\quad=\sum_{m=1}^{ \left\lfloor nT \right\rfloor} 
   \underbrace{\left( \frac{1}{\sqrt{n}}\mathcal{B}^N(t^n_{m-1})g(\sqrt{n}\Delta_m^n Y) -\sqrt{n}\int_{t_{m-1}^n}^{t_m^n} \mathcal{B}^N(s)\rho_g(X_s)ds\right)}_{\eta_m^n}
\end{align*}
into 
\[
 \sum_{m=1}^{ \left\lfloor nT \right\rfloor} (\eta_m^{n,1}+\eta_m^{n,2}+\eta_m^{n,3}) \, ,
\]
where 
 \begin{align*}
 \eta_m^{n,1}&=\frac{1}{\sqrt{n}}\mathcal{B}^N(t^n_{m-1})(g(\beta_m^n)-\rho_{m-1}^n(g))\, ,\\
 \eta_m^{n,2}&=\sqrt{n}\int_{t_{m-1}^n}^{t_m^n}(\mathcal{B}^N(t^n_{m-1})\rho_{m-1}^n(g)-\mathcal{B}^N(s)\rho_g(X_s))ds \, ,\\
  \eta_m^{n,3}&= \frac{1}{\sqrt{n}}\mathcal{B}^N(t^n_{m-1})(g(\sqrt{n}\Delta_m^n Y)- g(\beta_m^n)) \, ,
\end{align*}
and $\beta_m^n$ and $\rho_m^n$ are defined in~\eqref{eq:betamn} and~\eqref{eq:rhomn}.
We divide the proof into several steps: the \emph{first step}, which is subject of Proposition~\ref{prop:1} below, consists in dealing with
\begin{equation}\label{eq:firststepU}
  U^n_T(g,N):=\sum_{m=1}^{ \left\lfloor nT \right\rfloor}  \eta_m^{n,1}=\frac{1}{\sqrt{n}}\sum_{m=1}^{ \left\lfloor nT \right\rfloor} \mathcal{B}^N(t^n_{m-1})(g(\beta_{m}^n)- \rho_{m-1}^n(g)).
\end{equation}
As stated in Proposition~\ref{prop:1} below, $(U^n(g,N))_{n\in \mathbb{N}}$ converges stably in law to the process $U(g,N)$, defined in~\eqref{eq:UgN}.

For the central limit theorem the \emph{second step} consists in showing that
\[
\sum_{m=1}^{ \left\lfloor nT \right\rfloor}\eta_m^{n,2}=\sum_{m=1}^{ \left\lfloor nT\right\rfloor}\sqrt{n}\int_{t_{m-1}^n}^{t_m^n}(\mathcal{B}^N(t^n_{m-1})\rho_{m-1}^n(g)-\mathcal{B}^N(s)\rho_g(X_s))ds
\]
converges in probability to $0$. This can be shown similarly as in~\cite[Section 5.3.3, B]{JP:12}.
Note that for the consistency result it is enough that
\[
\sum_{m=1}^{ \left\lfloor nT \right\rfloor}\frac{1}{\sqrt{n}}\eta_m^{n,2}=\sum_{m=1}^{ \left\lfloor nT\right\rfloor}\frac{1}{n}\mathcal{B}^N(t^n_{m-1})\rho_{m-1}^n(g)-\int_0^T \mathcal{B}^N(s)\rho_g(X_s)ds.
\]
converges in probability to $0$, which is implied by Riemann integrability.

In the \emph{third step} we finally consider
\[
 \sum_{m=1}^{ \left\lfloor nT \right\rfloor} \eta_m^{n,3}= \sum_{m=1}^{ \left\lfloor nT \right\rfloor}\frac{1}{\sqrt{n}}\mathcal{B}^N(t^n_{m-1})(g(\sqrt{n}\Delta_m^n Y)- g(\beta_m^n)),
\]
which we split into $\sum_{m=1}^{ \left\lfloor nT \right\rfloor} \eta_m^{n,31}+ \eta_m^{n,32}$,
with
\begin{align*}
 \eta_m^{n,31}&=\frac{1}{\sqrt{n}}\mathcal{B}^N(t^n_{m-1})(g(\sqrt{n}\Delta_m^n D^Y(\chi'))- g(\beta_m^n)),\\
\eta_m^{n,32}&=\frac{1}{\sqrt{n}}\mathcal{B}^N(t^n_{m-1})(g(\sqrt{n}\Delta_m^n Y)- g(\sqrt{n}\Delta_m^n D^Y(\chi'))),
\end{align*}
where $D^Y(\chi')$ denotes the diffusion part of $Y$ defined in~\eqref{eq:diffpart}.
For the central limit theorem to hold true, $\sum_{m=1}^{ \left\lfloor nT \right\rfloor} \eta_m^{n,31}$ 
has to converge to $0$ in probability. To this end, it suffices to prove that
\begin{align}
\sum_{m=1}^{ \left\lfloor nT \right\rfloor}\mathbb{E}\left[ \|\eta_m^{n,31}\|^2 \right] &\to 0 \, ,\label{eq:secondmomenteta31}\\
 \sum_{m=1}^{ \left\lfloor nT \right\rfloor} \mathbb{E}\left[ \eta_m^{n,31}| \mathcal{F}_{t_{m-1}^n}\right] &\stackrel{\mathbb{P}}{\to} 0\label{eq:condmomenteta31}
\end{align}
holds true. Indeed, \eqref{eq:secondmomenteta31} implies (compare~\cite[Lemma 2.2.11]{JP:12})
\[
\sum_{m=1}^{ \left\lfloor nT \right\rfloor} \eta_m^{n,31}- \mathbb{E}\left[ \eta_m^{n,31}| \mathcal{F}_{t_{m-1}^n}\right]\stackrel{\mathbb{P}}{\to} 0
\]
and thus together with~\eqref{eq:condmomenteta31}, we have $\sum_{m=1}^{ \left\lfloor nT \right\rfloor} \eta_m^{n,31} \stackrel{\mathbb{P}}{\to} 0$.
Both requirements~\eqref{eq:secondmomenteta31} and~\eqref{eq:condmomenteta31} are met under the conditions on $X$, $Y$ and $g$ and the proof of~\eqref{eq:secondmomenteta31} can be found in~\cite[Proof of Theorem 5.1]{BNGJPS:06} and~\eqref{eq:condmomenteta31} is shown in~\cite[Section 5.3.3, C]{JP:12}. For the consistency result only~\eqref{eq:secondmomenteta31} has to be satisfied, which holds under the less restrictive assumptions $(H)$ and $(J)$ (see~\cite[Theorem 5.1]{BNGJPS:06}). Indeed, we only have to show that 
\[
\sum_{m=1}^{ \left\lfloor nT \right\rfloor} \frac{1}{\sqrt{n}}
\eta_m^{n,31} \stackrel{\mathbb{P}}{\to} 0,
\]
which follows from the Cauchy-Schwarz inequality and~\eqref{eq:secondmomenteta31}, since
\[
\sum_{m=1}^{ \left\lfloor nT \right\rfloor} \frac{1}{\sqrt{n}}
\mathbb{E}\left[\|\eta_m^{n,31}\|\right]\leq 
\left(T\sum_{m=1}^{ \left\lfloor nT \right\rfloor}\mathbb{E}\left[\|\eta_m^{n,31}\|^2\right]\right)^{\frac{1}{2}}  \to 0.
\]

Finally, according to the assumptions $(L(0))$ or $(L(\frac{1}{2}))$, respectively, we have for all $\varepsilon > 0$
\begin{align}\label{eq:jumppartcons}
\mathbb{P}\left[\left\|\sum_{m=1}^{ \left\lfloor nT \right\rfloor} \frac{1}{\sqrt{n}}\eta_m^{n,32}\right\| > \varepsilon\right]
&\leq \mathbb{P}\left[\tau_p \leq T\right]+\mathbb{P}\left[\left\|\sum_{m=1}^{ \left\lfloor nT \right\rfloor} 
\frac{1}{\sqrt{n}}\eta_m^{n,32}(p)\right\| > \varepsilon\right],
\end{align}
and 
\begin{align}\label{eq:jumppartclt}
\mathbb{P}\left[\left\|\sum_{m=1}^{ \left\lfloor nT \right\rfloor} \eta_m^{n,32}\right\| > \varepsilon\right]
& \leq \mathbb{P}\left[\tau_p \leq T\right]+\mathbb{P}\left[\left\|\sum_{m=1}^{ \left\lfloor nT \right\rfloor} 
\eta_m^{n,32}(p)\right\| > \varepsilon\right],
\end{align}
where 
\[
\eta_m^{n,32}(p)=\mathcal{B}^N(t^n_{m-1})(g(\sqrt{n}\Delta_m^n Y(p))- g(\sqrt{n}\Delta_m^n D^{Y(p)}(\chi'))).
\]
and $Y(p)$ is defined in~\eqref{eq:Yloc}. Due to $(L(0))$ or $(L(\frac{1}{2}))$, respectively, the second term on the right hand side of~\eqref{eq:jumppartcons} and~\eqref{eq:jumppartclt} respectively tends to $0$ as $n \to \infty$ for all $p$.
Since  $\mathbb{P}\left[\tau_p \leq T\right] \to 0$ as $p \to \infty$, we deduce
\[
\mathbb{P}\left[\left\|\sum_{m=1}^{ \left\lfloor nT \right\rfloor} \frac{1}{\sqrt{n}}\eta_m^{n,32}\right\| > \varepsilon\right] \to 0
\]
and
\[
\mathbb{P}\left[\left\|\sum_{m=1}^{ \left\lfloor nT \right\rfloor} \eta_m^{n,32}\right\| > \varepsilon\right] \to 0
\]
respectively, which completes the proof.
\end{proof}

The following proposition is an application of~\cite[Proposition 4.1]{BNGJPS:06} or~\cite[Theorem 4.2.1]{JP:12}. The proof is omitted as it works along the lines of~\cite[Proposition 4.1]{BNGJPS:06}.
Again we here only consider the one-dimensional case, i.e., $Y$ and $X$ are one-dimensional and $g: \mathbb{R} \to \mathbb{R}$.

\begin{proposition}\label{prop:1}
 Assume that the process $X$ is c\`adl\`ag and bounded by a constant. Let $g$ be a function of at most polynomial growth.
 Then the sequence of processes $(U^n_T(g,N))_{n \in \mathbb{N}}$ defined in~\eqref{eq:firststepU} is $C$-tight. Moreover, if $g$ is even, 
 then it converges stably in law to the process $U(g,N)$ given componentwise for $k \in \{-N, \ldots, N\}$ by
 \begin{align}\label{eq:limit_proc}
  U(g,N)_T^k=\sum_{k'=1}^{2(2N+1)}\int_0^T \delta_{s}^{kk'} dW^{k'}_{s},
 \end{align}
where
\[
 \sum_{r=1}^{2(2N+1)} \delta^{kr}_{s}\overline{\delta^{k'r}_{s}}=(\rho_{g^2}(X_s)-(\rho_{g}(X_s))^2)e^{-\im\frac{2\pi}{T}(k-k')s}
\]
and $W$ is a $2(2N+1)$-dimensional Brownian motion which is defined on an extension of the probability space $(\Omega, \mathcal{F}, (\mathcal{F}_t)_{t \geq 0}, \mathbb{P})$ and is independent
of the $\sigma$-field $\mathcal{F}$.
\end{proposition}

\subsection{Jump Robust Estimators - Proof of Corollary~\ref{cor:jumprobust}}\label{sec:jumpspec}

The aim of this section is to specify classes of functions $g$ and conditions on the jumps such that condition $(L(0))$ or $(L(\frac{1}{2}))$, respectively, is satisfied and such that the estimator $V(Y,g,k)_T^n$ given in~\eqref{eq:Fouriercofest} is robust to jumps. In particular, Corollary~\ref{cor:jumprobust} is a consequence of the following proposition.

\begin{proposition}\phantomsection\label{prop:jumprobust}
\begin{enumerate}
\item\label{prop:jumprobustLLN} Let $g$ be continuous with $g(x)=o(\|x^2\|)$ as $\|x \| \to \infty$. Moreover, suppose that
\begin{align}\label{eq:boundcond}
\sup_{\omega} \int_0^T\left( \|b_t^Y\|+ \|X_t\|+ \int (\|\xi\|^2 \wedge 1)K_t(d\xi)\right) ds < \infty
\end{align}
and $\sup_{\omega, t\in[0,T]}\| \Delta X_t(\omega)\|< \infty$.
Then for all $m \in \{1, \ldots, \left\lfloor nt\right\rfloor\}$
\begin{align}\label{eq:L0}
\lim_{n \to \infty} \mathbb{E}\left[\left\|g(\sqrt{n}\Delta_m^n Y)- g(\sqrt{n}\Delta_m^n D^{Y}(\id))\right\|\right]=0.
\end{align}
\item\label{prop:jumprobustCLT} Suppose that $g$ satisfies for some $q\geq 0$ and some $0< r \leq r'<1$
\begin{align}\label{eq:growth}
\|g(x)-g(y)\|\leq C (1+\|y\|^q)(\|x-y\|^{r}+\|x-y\|^{r'}).
\end{align}
Moreover, let $\beta \in [0,2]$ and assume that for all $t \in [0,T]$
 \begin{align}
 \mathbb{E}\left[\int_{\|\xi\|\leq 1}\|\xi\|^{\beta} K_{t}(d\xi)\right]& < \infty,\label{eq:smalljumps}\\
 \mathbb{E}\left[\int_{\|\xi\|> 1}\|\xi\| K_{t}(d\xi)\right]&< \infty \label{eq:bigjumps}
 \end{align}
and that $b^Y(\chi 1_{\{\beta>1\}})=b^Y+\int (\chi(\xi)1_{\{\beta>1\}} -\chi(\xi)) K(d\xi)$ as defined in~\eqref{eq:moddrift} and $X$ is uniformly bounded.
\begin{enumerate}
\item If $\beta\leq r < 1$, then there exists some $\kappa \in (1, \frac{1}{r'}]$ such that
 \begin{align}\label{eq:estimatekappa}
  \mathbb{E} \left[\left\|g(\sqrt{n} \Delta_m^n Y)-g\left(\sqrt{n}\Delta_m^n D^Y(0))\right)\right\| \right]&\leq C \frac{1}{n^{\frac{1}{\kappa}-\frac{r'}{2}}},
 \end{align}
 for all $m \in \{1, \ldots, \left\lfloor nt\right\rfloor\}$.
 \item If $\beta > r$, then 
 \begin{align*}
  \mathbb{E} \left[\left\|g(\sqrt{n} \Delta_m^n Y)-g\left(\sqrt{n}\Delta_m^n D^Y(\chi 1_{\{\beta>1\}})\right)\right\| \right]&\leq C \frac{1}{n^{r\frac{2-\beta}{2\beta}}},
 \end{align*}
 for all $m \in \{1, \ldots, \left\lfloor nt\right\rfloor\}$.
 \end{enumerate}
\end{enumerate}
\end{proposition}

\begin{remark}\phantomsection\label{rem:locjump}
\begin{enumerate}
\item\label{rem:locjump1}
Note that we can localize $Y$ and consider processes $(Y(p))_p$ for which~\eqref{eq:boundcond} is automatically satisfied (see, e.g.,~\cite[Lemma 3.4.5]{JP:12}). Assertion (i) then holds true for $Y(p)$, which implies that $(L(0))$ is satisfied without further conditions on the characteristics of $Y$. Similarly the boundedness assumption on $b^Y(\chi1_{\{\beta >1\}})$ and $X$ in statement (ii) can also be obtained by localizing the original process.
\item Condition~\eqref{eq:growth} is satisfied by many classes of functions, e.g., globally H\"older continuous functions or $C^1$ functions which satisfy $\|g(x)\|\leq C(1+\|x\|^{r'})$ for $r'\in [0,1)$ and $\|\nabla g(x)\| \leq C(1+\|x\|^q)$ for some $q \geq 0$ (compare also~\cite[Remark 5.3.7]{JP:12}).
Let us also remark that, if~\eqref{eq:growth} is satisfied with $q=0$, $\kappa$ in~\eqref{eq:estimatekappa} can be chosen to be $1$.
\end{enumerate}
\end{remark}

\begin{proof}
The first assertion is proved in~\cite[Lemma 3.4.6]{JP:12}. Concerning (ii), we shall distinguish the cases $\beta \leq 1$ and $\beta >1$ and set without loss of generality $\chi(\xi)=1_{\{\|\xi\|\leq 1\}}\xi$. 
Due to the assumption on $g$, we have
\begin{equation}\label{eq:jumpestimate}
\begin{split}
 &\mathbb{E} \left[\left\|g(\sqrt{n} \Delta_m^n Y)-g\left(\sqrt{n}\Delta_m^n D^Y(\chi 1_{\{\beta>1\}})\right)\right\| \right]\\
 &\quad \leq 
  C\mathbb{E} \Bigg[(1+\|\sqrt{n}\Delta_m^n D^Y(\chi 1_{\{\beta>1\}})\|^q) \times \\
  &\qquad \qquad\times\Bigg( \Bigg\|\sqrt{n}\int_{t_{m-1}^n}^{t_{m}^n}\int_{\{\|\xi\| \leq 1\}} \xi(\mu^Y(d\xi,ds)-1_{\{\beta >1\}}K_{s}(d\xi)ds)\Bigg\|^r + \\
  &\qquad \qquad \quad+ \Bigg\|\sqrt{n}\int_{t_{m-1}^n}^{t_{m}^n}\int_{\{\|\xi\| > 1\}} \xi\mu^Y(d\xi,ds)\Bigg\|^{r} + \\
  &\qquad \qquad\quad+\Bigg\|\sqrt{n}\int_{t_{m-1}^n}^{t_{m}^n}\int_{\{\|\xi\| \leq 1\}} \xi(\mu^Y(d\xi,ds)-1_{\{\beta >1\}}K_{s}(d\xi)ds)\Bigg\|^{r'} + \\
 &\qquad \qquad\quad + \Bigg\|\sqrt{n}\int_{t_{m-1}^n}^{t_{m}^n}\int_{\{\|\xi\| > 1\}} \xi\mu^Y(d\xi,ds)\Bigg\|^{r'}\Bigg)\Bigg].
  \end{split}
  \end{equation}
  Applying H\"older's inequality, we get for $1<\iota,\kappa<\infty$ such that $\frac{1}{\iota}+\frac{1}{\kappa}=1$
\begin{align*}
&\mathbb{E}\left[\|\sqrt{n}\Delta_m^n D^Y(\chi 1_{\{\beta>1\}})\|^q J^{k}\right] \\
&\quad\leq 
\mathbb{E}\left[\|\sqrt{n}\Delta_m^n D^Y(\chi 1_{\{\beta>1\}})\|^{q\iota}\right]^{\frac{1}{\iota}}\mathbb{E}\left[J^{k\kappa}\right]^{\frac{1}{\kappa}}.
 \end{align*}
 Here, $J$ stands for each of the above expressions of the jumps and $k$ corresponds to $r$ or $r'$.
Due to our assumptions on $b^Y(\chi 1_{\{\beta>1\}})$ and $X$ and as a consequence of H\"older's, Jensen's and Burkholder-Davis-Gundy's inequality the first expectation is bounded by a constant $C_{q,\iota}$ for all $q >0$ and $1 <\iota <\infty$. 

Let us now consider for some $\alpha \in (0,1]$
\[
\mathbb{E}\left[\left\|\int_{t_{m-1}^n}^{t_{m}^n}\int_{\{\|\xi\| > 1\}} \xi\mu^Y(d\xi,ds)\right\|^{\alpha}\right].
\]
Then since, for $\alpha \in (0,1]$, $\|\sum_i x_i\|^{\alpha}\leq \sum_i \|x_i\|^{\alpha}$ and
due to~\eqref{eq:bigjumps}
\begin{align*}
\mathbb{E} \left[ {\left\|\int_{t_{m-1}^n}^{t_{m}^n}\int_{\{\|\xi\| > 1\}} \xi \mu^Y(d\xi,ds)\right\|}^{\alpha}\right]
&\leq  
\mathbb{E} \left[\int_{t_{m-1}^n}^{t_{m}^n}\int_{\{\|\xi\| > 1\}} \|\xi\|^{\alpha} \mu^Y(d\xi,ds)\right]\\
&=\mathbb{E} \left[\int_{t_{m-1}^n}^{t_{m}^n}\int_{\{\|\xi\| > 1\}} \|\xi\|^{\alpha} K_{s}(d\xi)ds\right]\\
&\leq C \frac{1}{n}.
\end{align*}

Consider now the term
\[
\mathbb{E} \left[\left\|\int_{t_{m-1}^n}^{t_{m}^n}\int_{\{\|\xi\| \leq 1\}} \xi(\mu^Y(d\xi,ds)-1_{\{\beta >1\}}K_{s}(d\xi)ds)\right\|^{\alpha}\right]
\]
for $\beta \leq 1$.
Then by H\"older's inequality and~\eqref{eq:smalljumps}, we have
\begin{align*}
\mathbb{E} \left[ {\left\|\int_{t_{m-1}^n}^{t_{m}^n}\int_{\{\|\xi\| \leq 1\}} \xi \mu^Y(d\xi,ds)\right\|}^{\alpha}\right]&=
 \left(\mathbb{E} \left[ {\left\|\sum_{t_{m-1}\leq s < t_m} \Delta_s Y 1_{\{\Delta_s Y \leq 1\}}\right\|}^{\alpha \vee \beta}\right]\right)^{\frac{\alpha}{\beta} \wedge 1}\\
&\leq  \left(\mathbb{E} \left[ \sum_{t_{m-1}\leq s < t_m}\left\| \Delta_s Y\right\|^{\alpha \vee \beta} 1_{\{\Delta_s Y \leq 1\}}\right]\right)^{\frac{\alpha}{\beta} \wedge 1}\\
&= \left( \mathbb{E} \left[\int_{t_{m-1}^n}^{t_{m}^n}\int_{\{\|\xi\| \leq 1\}} \|\xi\|^{\beta} K_{s}(d\xi)ds\right]\right)^{\frac{\alpha}{\beta} \wedge 1}\\
&\quad\leq C \frac{1}{n^{1\wedge\frac{\alpha}{\beta} }}.
\end{align*}

If $ \beta > 1$, we obtain a similar estimate by using H\"older and Burkholder-Davis-Gundy's inequality, the fact that
$\frac{\beta}{2} \leq 1$ and~\eqref{eq:smalljumps}:
\begin{align*}
 & \mathbb{E}\left[ \left\|\int_{t_{m-1}^n}^{t_{m}^n}\int_{\{\|\xi\| \leq 1\}}\xi(\mu^Y(d\xi,ds)-K_{s}(d\xi)ds)\right\|^{\alpha}\right]\\
 & \quad \leq   \left(\mathbb{E}\left[ \left\|\int_{t_{m-1}^n}^{t_{m}^n}\int_{\{\|\xi\| \leq 1\}}\xi(\mu^Y(d\xi,ds)-K_{s}(d\xi)ds)\right\|^{\beta}\right]\right)^{\frac{\alpha}{\beta}}\\
 & \quad \leq   \left(\mathbb{E}\left[ \left(\int_{t_{m-1}^n}^{t_{m}^n}\int_{\{\|\xi\| \leq 1\}}\|\xi\|^2\mu^Y(d\xi,ds)\right)^{\frac{\beta}{2}}\right]\right)^{\frac{\alpha}{\beta}}\\
 &\quad \leq    \left(\mathbb{E}\left[ \left(\sum_{t_{m-1}\leq s < t_m}\left\| \Delta_s Y\right\|^{\beta} 1_{\{\Delta_s Y \leq 1\}}\right)\right]\right)^{\frac{\alpha}{\beta}}\\
  &\quad \leq   \left(\mathbb{E}\left[\int_{t_{m-1}^n}^{t_{m}^n}\int_{\{\|\xi\| \leq 1\}} \|\xi\|^{\beta} K_s(d\xi)ds \right]\right)^{\frac{\alpha}{\beta}}\\
  &\quad \leq C  \frac{1}{n^{ \frac{\alpha}{\beta}}}= C \frac{1}{n^{ \frac{\alpha}{\beta} \wedge 1}}.
\end{align*}
The last equality follows from the fact that for $\beta >1$, $\frac{\alpha}{\beta} <1$.
Using these inequalities and setting $\alpha$ equal to $r\kappa$ and $r'\kappa'$, respectively, for some $1 < \kappa \leq \frac{1}{r}$ and $1 < \kappa' \leq \frac{1}{r'}$, we can estimate~\eqref{eq:jumpestimate} by
\[
 \mathbb{E} \left[\left\|g(\sqrt{n} \Delta_m^n Y)-g\left(\sqrt{n}\Delta_m^n D^Y(\chi 1_{\{\beta>1\}})\right)\right\| \right]\leq
C \frac{1}{n^{\min((\frac{1}{\kappa}\wedge \frac{r}{\beta})-\frac{r}{2},(\frac{1}{\kappa'}\wedge \frac{r'}{\beta})-\frac{r'}{2})}}.
\]
If $\beta\leq r$, we can choose some $\kappa \in (1,\frac{1}{r'}]$ such that this expression is simplified to
\[
 \mathbb{E} \left[\left\|g(\sqrt{n} \Delta_m^n Y)-g\left(\sqrt{n}\Delta_m^n D^Y(\chi 1_{\{\beta>1\}})\right)\right\| \right]\leq
C \frac{1}{n^{\frac{1}{\kappa}-\frac{r'}{2}}}
\]
and if $\beta >r$, we obtain
\[
 \mathbb{E} \left[\left\|g(\sqrt{n} \Delta_m^n Y)-g\left(\sqrt{n}\Delta_m^n D^Y(\chi 1_{\{\beta>1\}})\right)\right\| \right]\leq
C \frac{1}{n^{r\frac{(2-\beta)}{2\beta}}}.
\]
\end{proof}

\subsubsection{Proof of Corollary~\ref{cor:jumprobust}}

Using Proposition~\ref{prop:jumprobust} above, we are now prepared to prove Corollary~\ref{cor:jumprobust}.

\begin{proof}
Assertion (i) is a direct consequence of Theorem~\ref{th:Fouriercofestconv}, Condition $(L(0))$, Proposition~\ref{prop:jumprobust}~\ref{prop:jumprobustLLN} and Remark~\ref{rem:locjump}~\ref{rem:locjump1}.

Again, in view of Theorem~\ref{th:Fouriercofestconv}, the proof of the second statement consists in verifying Condition $(L(\frac{1}{2}))$.
Since $\beta < 1$, we can consider $b^Y(0)=b^Y-\int \chi(\xi)K(d\xi)$
which is -- since it is a c\`agl\`ad process -- locally bounded.
 By the localization procedure (see~\cite[Lemma 4.4.8]{JP:12}) we can therefore consider processes $(Y(p))_p$, for which $b^{Y(p)}(0)$ and the diffusion characteristic $X$ are uniformly bounded 
and~\eqref{eq:bigjumps} and~\eqref{eq:smalljumps} for $\beta<1$ are satisfied, as required in Proposition~\ref{prop:jumprobust}~\ref{prop:jumprobustCLT}.

This proposition then yields in the case $\frac{\beta}{2-\beta}<\beta\leq r$
\begin{align}\label{eq:firstcase}
  \mathbb{E} \left[\left\|g(\sqrt{n} \Delta_m^n Y(p))-g\left(\sqrt{n}\Delta_m^n D^{Y(p)}(0)\right)\right\| \right]\leq C \frac{1}{n^{\frac{1}{\kappa}-\frac{r'}{2}}},
  \end{align}
for some $\kappa \in (1, \frac{1}{r'}]$ and in the case $\frac{\beta}{2-\beta}<r< \beta$
\[
  \mathbb{E} \left[\left\|g(\sqrt{n} \Delta_m^n Y(p))-g\left(\sqrt{n}\Delta_m^n D^{Y(p)}(0)\right)\right\| \right]\leq C \frac{1}{n^{r\frac{2-\beta}{2\beta}}}.
  \]
By choosing $1 <\kappa < \frac{2}{1+r'} < \frac{1}{r'}$ in~\eqref{eq:firstcase},  $(L(\frac{1}{2}))$  is satisfied in both cases
and the assertion follows.
\end{proof}

\section{Proofs of the Asymptotic Properties of the Estimator for the Instantaneous Covariance Process}\label{sec:spotvar}

In this section we provide the proofs of Theorem~\ref{th:consistency} and Theorem~\ref{th:Fourierestconv}.
For the study of the asymptotic properties of the instantaneous covariance estimator $\widehat{\rho_{g}(X)}^{n,N}_{t}$ given in~\eqref{eq:estimatorrhoX}, that is,
\[
\widehat{\rho_{g}(X)}^{n,N}_{t}:=\frac{1}{T}\sum_{k=-N}^{N} \left(1-\frac{|k|}{N}\right)e^{\im \frac{2\pi}{T} kt}V(Y,g,k)_T^{n}
\]
we need to analyze two different errors, namely
\begin{enumerate}
\item
the error which comes from the fact that we use estimators for the Fourier coefficients instead of the true quantities, that is,
\begin{align}\label{eq:staterror}
E^{n,N}(t):=\widehat{\rho_{g}(X)}^{n,N}_{t}-\sum_{k=-N}^{N} \left(1-\frac{|k|}{N}\right)e^{\im \frac{2\pi}{T} kt}\mathcal{F}(\rho_g(X))(k),
\end{align}
\item the error which we make by truncating the Fourier-F\'ejer sum, that is,
\begin{align}\label{eq:E0}
E^N_0(t):=\sum_{k=-N}^{N} \left(1-\frac{|k|}{N}\right)e^{\im \frac{2\pi}{T} kt}\mathcal{F}(\rho_g(X))(k)-\frac{\rho_g(X_{t-})+\rho_g(X_t)}{2}.
\end{align}
\end{enumerate}

The term $E^N_0(t)$ can be treated with well known deterministic results on Fourier-F\'ejer series, whereas the statistical error $E^{n,N}(t)$ needs to be decomposed in several parts and handled with probabilistic methods, in particular limit theorems for triangular arrays. To apply these methods let us first remark that $\widehat{\rho_{g}(X)}^{n,N}_{t}$ can be written as
\begin{align}\label{eq:estFejerkernel}
\widehat{\rho_{g}(X)}^{n,N}_{t}=\frac{1}{T}\sum_{m=1}^{ \left\lfloor nT \right\rfloor}  \frac{1}{n} F_N\left(\frac{2\pi}{T}(t-t_{m-1}^n)\right)g(\sqrt{n}\Delta_m^n Y),
\end{align}
where $F_N$ denotes the F\'ejer kernel defined by
\begin{align}\label{eq:Fejerkernel}
F_N(x):=\frac{1}{N+1}\frac{\sin\left((N+1)\frac{x}{2}\right)^2}{\sin\left(\frac{x}{2}\right)^2}=
\sum_{k=-N}^{k=N}\left(1-\frac{|k|}{N}\right) e^{\im kx}.
\end{align}
Using this representation, the term $E^{n,N}(t)$ given in~\eqref{eq:staterror} can now be written as
\begin{align*}
\frac{1}{T}\left(\sum_{m=1}^{ \left\lfloor nT \right\rfloor}  \frac{1}{n} F_N\left(\frac{2\pi}{T}(t-t_{m-1}^n)\right)g(\sqrt{n}\Delta_m^n Y)-\int_0^T  F_N\left(\frac{2\pi}{T}(t-s)\right)\rho_g(X_s)ds\right).
\end{align*}
This can be further decomposed into $\sum_{i=1}^3 E^{n,N}_i(t)$, where
\begin{align}
E^{n,N}_1(t)&:=\frac{1}{T}\sum_{m=1}^{ \left\lfloor nT \right\rfloor} \frac{1}{n} F_N\left(\frac{2\pi}{T}(t-t_{m-1}^n)\right)(g(\sqrt{n}\Delta_m^n Y)- g(\beta_m^n)),\label{eq:E1}\\
E^{n,N}_2(t)&:=\frac{1}{T}\sum_{m=1}^{ \left\lfloor nT \right\rfloor}  \frac{1}{n} F_N\left(\frac{2\pi}{T}(t-t_{m-1}^n)\right)(g(\beta_m^n)-\rho_{m-1}^n(g)),\label{eq:E2}\\
E^{n,N}_3(t)&:=\frac{1}{T}\Bigg(\sum_{m=1}^{ \left\lfloor nT \right\rfloor}  \frac{1}{n} F_N\left(\frac{2\pi}{T}(t-t_{m-1}^n)\right)\rho_{m-1}^n(g)\\
&\quad-\int_0^T  F_N\left(\frac{2\pi}{T}(t-s)\right)\rho_{g}(X_s)ds\Bigg)\label{eq:E3},
\end{align}
with $\beta_m^n$ and $\rho_m^n$ given in~\eqref{eq:betamn} and~\eqref{eq:rhomn}.
In view of this decomposition, we start with the following lemma which deals with convergence of Riemann sums for the F\'ejer kernel and which is crucial for studying asymptotic properties of the instantaneous covariance estimators.

\begin{lemma}\label{lem:Fejerconv}
Let $n, N \in \mathbb{N}$ and denote
\[
x_0^n:=-\pi<x_1^n:=-\pi+\frac{1}{n}<\cdots<x_m^n:=-\pi+\frac{m}{n}<x_{\left\lfloor n 2 \pi \right\rfloor }^n=-\pi+\frac{\left\lfloor n 2 \pi \right\rfloor }{n}
\]
and consider the F\'ejer kernel $F_N$ given in~\eqref{eq:Fejerkernel}.
Moreover, let $ \gamma > 1$ and suppose that  $\lim \frac{n}{N^{\gamma}}=K$ for some constant $K >0$ and let $h$ be a $\delta$-H\"older continuous function with $\delta \in (0,1]$.
Then, the following identities holds true:
\begin{align}
\lim_{n,N \to \infty} \sum_{m=1}^{\left\lfloor n 2 \pi \right\rfloor} \frac{1}{n}F_N(x_{m-1}^n)&=\lim_{N \to \infty}\int_{-\pi}^{\pi} F_N(x)dx=2 \pi,\label{eq:FN}\\
\lim_{n,N \to \infty} \sum_{m=1}^{\left\lfloor n 2 \pi \right\rfloor} \frac{1}{n} \frac{F_N^2(x_{m-1}^n)}{N}&=
\lim_{N \to \infty}\int_{-\pi}^{\pi} \frac{F_N^2(x)}{N}dx=
\frac{4\pi}{3},\label{eq:FN2}\\
\lim_{n,N \to \infty} \sum_{m=1}^{\left\lfloor n 2 \pi \right\rfloor} \frac{1}{n} \frac{F_N^2(y-x_{m-1}^n)}{N}h(x_{m-1}^n)&=
\lim_{N \to \infty}\int_{-\pi}^{\pi} \frac{F_N^2(y-x)}{N}h(x)dx=\frac{4\pi}{3}h(y), \label{eq:FN3}
\end{align}
Furthermore, we have the following error estimates:
\begin{align*}
\left|  \sum_{m=1}^{\left\lfloor n 2 \pi \right\rfloor} \frac{1}{n} F_N(x_{m-1}^n)-\int_{-\pi}^{\pi} F_N(x)dx\right|& \leq C \frac{N}{n}=\frac{C}{K}\frac{1}{N^{\gamma-1}},\\
\left|  \sum_{m=1}^{\left\lfloor n 2 \pi \right\rfloor} \frac{1}{n} h(x_{m-1}^n)F_N(x_{m-1}^n)-\int_{-\pi}^{\pi}h(x) F_N(x)dx\right| &\leq C \max\left(\frac{N}{n},\frac{1}{n^{\delta}}\right)\\
&=\widetilde{C}\max\left(\frac{1}{N^{\gamma-1}},\frac{1}{N^{\gamma\delta}}\right).
\end{align*}
\end{lemma}

\begin{remark}\phantomsection\label{rem:Fejerconv}
\begin{enumerate}
\item\label{rem:Fejerconv1}
In the sequel we shall consider analogues of expressions~\eqref{eq:FN2} and~\eqref{eq:FN3} on the interval $[0,T]$, that is, 
\begin{align*}
\lim_{N \to \infty}\int_{0}^{T} \frac{F_N^2(\frac{2\pi}{T}x)}{N}dx&=\frac{2 T}{3},\\
\lim_{N \to \infty}\int_{0}^{T} \frac{F_N^2(\frac{2\pi}{T}(y-x))}{N}h(x)dx&=\frac{2 T}{3}h(y),
\end{align*}
which can be derived similarly as in the proof below.
\item The expression $F_N(x)^2/N$ appropriately normalized
corresponds to the so-called Jackson kernel (see, e.g.,~\cite[Section 4.2]{L:86}).
\end{enumerate}
\end{remark}

\begin{proof}
We have convergence of the Riemann sums in~\eqref{eq:FN}
to the corresponding integral if the grid becomes finer within the zeros of $F_N$. Since the distance between $2$ zeros is 
$\frac{2\pi}{N+1}$ ($\frac{4\pi}{N+1}$ for the $2$ zeros closest to the origin), this is the case if only if $\gamma > 1$.
By the same argument the left hand side of~\eqref{eq:FN2} and~\eqref{eq:FN3} converges to the corresponding integrals
if and only if $\gamma > 1$. The assertion then follows from the following calculation
\begin{align*}
&\int_{-\pi}^{\pi}\frac{1}{N}\sum_{k,l}\left(1-\frac{|k|}{N}\right)\left(1-\frac{|l|}{N}\right) e^{\im kx}e^{\im lx}dx\\
&\quad=\frac{1}{N}\sum_{k,l}\left(1-\frac{|k|}{N}\right)\left(1-\frac{|l|}{N}\right) \int_{-\pi}^{\pi}e^{\im (k+l)x}dx\\
&\quad =2\pi\frac{1}{N}\sum_{k}\left(1-\frac{|k|}{N}\right)^2\\
&\quad =2\pi\frac{2N^2+1}{3N^2} \to \frac{4 \pi}{3}.
\end{align*}

Concerning the error estimates, we have due to the mean value theorem
\begin{align*}
&\left|  \sum_{m=1}^{\left\lfloor n 2 \pi \right\rfloor} \frac{1}{n} F_N(x_{m-1}^n)-\int_{-\pi}^{\pi} F_N(x)dx\right| \\
&\quad =\left|  \sum_{m=1}^{\left\lfloor n 2 \pi \right\rfloor} \frac{1}{n} \left(F_N(x_{m-1}^n)-  F_N(y_{m-1}^n)\right)\right|\\
&\quad\leq \sum_{m=1}^{\left\lfloor n 2 \pi \right\rfloor} \frac{1}{n} \left|F_N(x_{m-1}^n)-  F_N(y_{m-1}^n)\right|\\
&\quad\leq \sum_{k=-\frac{(N+1)}{2}}^{\frac{(N-1)}{2}} \sum_{x_m \in [\frac{2\pi k}{N+1}, \frac{2\pi (k+1)}{N+1}]}\frac{1}{n} \left|F_N(x_{m-1}^n)-  F_N(y_{m-1}^n)\right|
\end{align*}
with some $y_{m-1}^n \in [x_{m-1}^n, x_m^n]$. Using again the mean value theorem, we can further estimate
\begin{align*}
&\sum_{x_m \in [\frac{2\pi k}{N+1}, \frac{2\pi (k+1)}{N+1}]}\frac{1}{n} \left|F_N(x_{m-1}^n)-  F_N(y_{m-1}^n)\right|
\\&\quad\leq\sum_{\eta_{m-1}^n \in [x_{m-1}^n, y_{m-1}^n],\, x_m \in [\frac{2\pi k}{N+1}, \frac{2\pi (k+1)}{N+1}]}\frac{1}{n^2} \left|F_N'(\eta_{m-1}^n)\right|\\
&\quad\leq \max_{\eta \in [\frac{2\pi k}{N+1}, \frac{2\pi (k+1)}{N+1}]}\left|F_N'(\eta)\right| \frac{2\pi n}{N+1}\frac{1}{n^2}.
\end{align*}
Since
\[
\max_{\eta \in [\frac{2\pi k}{N+1}, \frac{2\pi (k+1)}{N+1}]}\left|F_N'(\eta)\right|\frac{1}{N}\leq C\max_{x \in [\frac{2\pi k}{N+1}, \frac{2\pi (k+1)}{N+1}]}\left|F_N(x)\right|,
\]
and since
\[
\max_{x \in [\frac{2\pi k}{N+1}, \frac{2\pi (k+1)}{N+1}]}\left|F_N(x)\right|\frac{1}{N} \leq C \int_{\frac{2\pi k}{N+1}}^{\frac{2\pi (k+1)}{N+1}} F_N(x)dx,
\]
it follows that
\[
\max_{\eta \in [\frac{2\pi k}{N+1}, \frac{2\pi (k+1)}{N+1}]}\left|F_N'(\eta)\right|\leq C N^2\int_{\frac{2\pi k}{N+1}}^{\frac{2\pi (k+1)}{N+1}} F_N(x)dx.
\]
Hence, we have
\begin{align*}
&\sum_{k=-\frac{(N+1)}{2}}^{\frac{(N-1)}{2}} \sum_{x_m \in [\frac{2\pi k}{N+1}, \frac{2\pi (k+1)}{N+1}]}\frac{1}{n} \left|F_N(x_{m-1}^n)-  F_N(y_{m-1}^n)\right|\\
&\quad\leq\frac{2\pi n}{N+1}\frac{1}{n^2}\sum_{k=-\frac{(N+1)}{2}}^{\frac{(N-1)}{2}} \max_{\eta \in [\frac{2\pi k}{N+1}, \frac{2\pi (k+1)}{N+1}]}\left|F_N'(\eta)\right|\\
&\quad\leq \frac{2\pi n}{N+1}\frac{1}{n^2} \sum_{k=-\frac{(N+1)}{2}}^{\frac{(N-1)}{2}}C N^2\int_{\frac{2\pi k}{N+1}}^{\frac{2\pi (k+1)}{N+1}} F_N(x)dx\\
&\quad\leq C  \frac{2\pi n}{N+1}\frac{1}{n^2} N^2 \underbrace{\int_{-\pi}^{\pi} F_N(x) dx}_{=2\pi}\\
&\quad\leq C \frac{N}{n}=C \frac{N^{\gamma}}{n}\frac{1}{N^{\gamma-1}},
\end{align*}
which yields the assertion.
Concerning
\[
\left|  \sum_{m=1}^{\left\lfloor n 2 \pi \right\rfloor} \frac{1}{n} h(x_{m-1}^n)F_N(x_{m-1}^n)-\int_{-\pi}^{\pi}h(x) F_N(x)dx\right| 
\]
we can estimate it by
\begin{align*}
&\left|  \sum_{m=1}^{\left\lfloor n 2 \pi \right\rfloor} \frac{1}{n}h(x_{m-1}^n) F_N(x_{m-1}^n)-\int_{x_{m-1}^n}^{x_m} h(x_{m-1}^n)F_N(x)dx\right| + \\
 &\quad+\left|  \sum_{m=1}^{\left\lfloor n 2 \pi \right\rfloor} \int_{x_{m-1}^n}^{x_m^n} (h(x_{m-1}^n)-h(x))F_N(x)dx\right|\\
 &\leq \max_{m} |h(x_{m-1}^n)| \sum_{m=1}^{\left\lfloor n 2 \pi \right\rfloor}\left| \frac{1}{n} F_N(x_{m-1}^n)-\int_{x_{m-1}^n}^{x_m} F_N(x)dx\right| + \\
 &\quad+\frac{C}{n^{\delta}} \int_{-\pi}^{\pi} F_N(x)dx\\
 &\leq C \left(\frac{N}{n}+\frac{1}{n^{\delta}}\right)\\
 &\leq C \max\left(\frac{N}{n},\frac{1}{n^{\delta}}\right)\\
 &=\widetilde{C}\max\left(\frac{1}{N^{\gamma-1}},\frac{1}{N^{\gamma\delta}}\right).
\end{align*}
\end{proof}

\subsection{Consistency - Proof of Theorem~\ref{th:consistency}}

Using the above lemma we can now proceed to establish consistency of the estimator given in~\eqref{eq:estimatorrhoX}.

\begin{proof}
Similar as in~\cite{BNGJPS:06} or~\cite[Section 4.4.1]{JP:12} and the proof of Theorem~\ref{th:Fouriercofestconv}, we strengthen assumption $(H)$ together with Condition~\eqref{eq:moddrift} to boundedness of $b^Y(\chi')$ and $X$.
As already explained in the introduction of this section, we decompose
 \[
 \left(\widehat{\rho_{g}(X)}^{n,N}_{t}-\frac{\rho_{g}(X_{t-})+\rho_{g}(X_t)}{2}\right) 
 \] 
 into
$E^N_0(t)+\sum_{i=1}^3 E^{n,N}_i(t)$, where $E^N_0(t)$ is given in~\eqref{eq:E0} and $E^{n,N}_i(t)$, $i=1,2,3$ in~\eqref{eq:E1} -~\eqref{eq:E3}.

By F\'ejer's theorem the term $\|E^N_0(t)\|$ converges to $0$ a.s., since $X$ is supposed to have c\`adl\`ag paths.
 
As a consequence of the proof of Theorem~\ref{th:Fourierestconv} below, the term $E^{n,N}_2(t)$ converges to $0$ in probability under the assumptions ($H$) and ($J$). 

Again by the c\`adl\`ag property and the boundedness of $t \to \rho_{g}(X_t)$ (recall the boundedness condition on $X$ and the fact that $g$ has at most polynomial growth) and the assumption $\gamma >1$, $\|E^{n,N}_3(t)\|$ converges a.s. to $0$ by Riemann integrability (cf. Lemma~\ref{lem:Fejerconv}).

Finally we have to focus on the $E^{n,N}_1(t)$, which we decompose into
\begin{align*}
TE^{n,N}_1(t)&= \sum_{m=1}^{ \left\lfloor nT \right\rfloor} \frac{1}{n} F_N\left(\frac{2\pi}{T}(t-t_{m-1}^n)\right)(g(\sqrt{n}\Delta_m^n Y)- g(\sqrt{n}\Delta_m^n D^Y(\chi')))\\
 &\quad+
 \sum_{m=1}^{ \left\lfloor nT \right\rfloor} \frac{1}{n} F_N\left(\frac{2\pi}{T}(t-t_{m-1}^n)\right)(g(\sqrt{n}\Delta_m^n D^Y(\chi'))- g(\beta_m^n)), 
\end{align*}
where $D^Y(\chi')$ is defined in~\eqref{eq:diffpart}.
The second term converges in probability to $0$,  
since it can be estimated by
\begin{align*}
 \sup_m\mathbb{E}\left[\|g(\sqrt{n}\Delta_m^n D^Y(\chi'))- g(\beta_m^n)\|\right] \sum_{m=1}^{ \left\lfloor nT \right\rfloor} \frac{1}{n} F_N\left(\frac{2\pi}{T}(t-t_{m-1}^n)\right)
\end{align*}
and we have  $\sup_m\mathbb{E}\left[\|g(\sqrt{n}\Delta_m^n D^Y(\chi'))- g(\beta_m^n)\|\right] \to 0$ 
(see the proof of~\cite[Lemma 5.3, Lemma 5.4]{BNGJPS:06}).
Writing $\zeta_m^n=(g(\sqrt{n}\Delta_m^n Y)- g(\sqrt{n}\Delta_m^n D^Y(\chi')))$ and $\zeta_m^n(p)=(g(\sqrt{n}\Delta_m^n Y(p))- g(\sqrt{n}\Delta_m^n D^Y(p)(\chi')))$, where $Y(p)$ is defined in~\eqref{eq:Yloc}, we have
\begin{align}\label{eq:Fourierconsjump}
&\mathbb{P}\left[\left\|\sum_{m=1}^{ \left\lfloor nT \right\rfloor} \frac{1}{n} F_N\left(\frac{2\pi}{T}(t-t_{m-1}^n)\right) \zeta_m^n\right\| > \varepsilon\right] \\&\quad\leq \mathbb{P}\left[\tau_p \leq T\right]+\mathbb{P}\left[\left\|\sum_{m=1}^{ \left\lfloor nT \right\rfloor} \frac{1}{n} F_N\left(\frac{2\pi}{T}(t-t_{m-1}^n)\right) \zeta_m^n(p)\right\| > \varepsilon\right]\notag.
\end{align}
Since $\mathbb{P}\left[\tau_p \leq T\right] \to 0$ as $p\to \infty$ and since $\mathbb{E}\left[\|\sum_{m=1}^{ \left\lfloor nT \right\rfloor} \frac{1}{n} F_N\left(\frac{2\pi}{T}(t-t_{m-1}^n)\right) \zeta_m^n(p)\|\right]$ can be estimated by
\[
\sup_m\mathbb{E}\left[\|\zeta_m^n(p)\|\right] \sum_{m=1}^{ \left\lfloor nT \right\rfloor} \frac{1}{n} F_N\left(\frac{2\pi}{T}(t-t_{m-1}^n)\right),
\]
which converges to $0$ for all $p$ due to Assumption $(L(0))$, ~\eqref{eq:Fourierconsjump} tends to $0$ as well.
\end{proof}

\subsection{Central Limit Theorem - Proof of Theorem~\ref{th:Fourierestconv}}

This section is dedicated to the proof of the central limit theorem for (functions of) the instantaneous covariance. 

\begin{proof}
Similarly as in the proof of Theorem~\ref{th:Fouriercofestconv}, we strengthen the assumption $(H1)$ to $(SH1)$, that is,
$b^Y(\chi')$, defined in~\eqref{eq:moddrift}, $X$ and the differential characteristics of the drift and the diffusion part of $\sqrt{X}$ and~\eqref{eq:comproot} are bounded by a constant. Analogously to the proof of
Theorem~\ref{th:consistency} we decompose 
\[
\sqrt{\frac{nT}{N}}\left(\widehat{\rho_{g}(X)}^{n,N}_{t}-\rho_{g}(X)\right) 
\]
into $\sqrt{\frac{nT}{N}} (\widetilde{E}_0^N(t)+\sum_{i=1}^3  E^{n,N}_i(t))$, where $E^{n,N}_i(t)$, $i=1,2,3$ are defined in~\eqref{eq:E1} -~\eqref{eq:E3} and 
$\widetilde{E}_0^N(t)$ is here given by 
\[
\widetilde{E}_0^N(t)=\sum_{k=-N}^{N} \left(1-\frac{|k|}{N}\right)e^{\im \frac{2\pi}{T} kt}\mathcal{F}(\rho_g(X))(k)-\rho_g(X_t).
\]
Denoting by $A_N$ the set
\[
A_N=\left\{ X \textrm{ jumps in } \left[t-\frac{T}{N}, t+\frac{T}{N}\right]\right\},
\]
 we have the following estimate for $\mathbb{P}\left[\sqrt{\frac{nT}{N}}\|\widetilde{E}^N_0(t)\| > \varepsilon\right]$:
\begin{align}\label{eq:decjump}
\mathbb{P}\left[\sqrt{\frac{nT}{N}}\|\widetilde{E}^N_0(t)\| > \varepsilon\right] \leq \mathbb{P}\left[A_N\right]+\mathbb{P}\left[\sqrt{\frac{nT}{N}}\|\widetilde{E}^N_0(t)\|1_{A^c_N} > \varepsilon\right].
\end{align}
Due to the assumption that $X$ has no fixed time of discontinuity, the first term converges to $0$. 
By the assumption of finite activity jumps and H\"older continuity of $t \mapsto X_t$ between two jumps, we have 
\[
\sqrt{\frac{nT}{N}}\|\widetilde{E}^N_0(t)(\omega)\|1_{A^c_N}(\omega) \leq C_{\omega} \sqrt{\frac{n}{N}}N^{-\delta}=\widetilde{C}_{\omega}N^{-\frac{1+2\delta-\gamma}{2}},
\]
for some finitely valued number $C_{\omega}$ (depending on $\omega$) (compare~\cite[Eq.~13]{MM:09} and~\cite{Z:49}). Since  $\gamma < 1+2\delta $ by assumption, the second term in the above decomposition thus also converges to $0$.

Due to Lemma~\ref{lem:Fejerconv} and again the assumption of finite activity jumps and H\"older continuity of $t \mapsto X_t$ between two jumps ,  $\sqrt{\frac{nT}{N}}\|E^{n,N}_3(t)\|1_{A^c_N}$ can be estimated by
\[
C_{\omega}\sqrt{\frac{n}{N}}\max\left(\frac{N}{n},\frac{1}{n^{\delta}}\right)=
 \widetilde{C}_{\omega} N^{\frac{\gamma-1}{2}}\max\left(\frac{1}{N^{\gamma-1}}, \frac{1}{N^{\delta\gamma}}\right),
\]
which converges to $0$, since $\frac{\gamma -1}{2}< \min(\gamma -1, \delta \gamma)$ again as a consequence of the assumption  $\gamma < 1+2\delta $. A similar decomposition as in~\eqref{eq:decjump} yields $ \sqrt{\frac{nT}{N}}\|E^{n,N}_3(t)\| \stackrel{\mathbb{P}}{\to} 0$.

Let us now consider $\sqrt{\frac{nT}{N}}E^{n,N}_1(t)$, which we decompose into
\begin{align*}
 &\sqrt{\frac{n}{NT}}\sum_{m=1}^{ \left\lfloor nT \right\rfloor} \frac{1}{n}F_N\left(\frac{2\pi}{T} (t-t_{m-1}^n)\right)
(g(\sqrt{n}\Delta_m^n Y)- g(\sqrt{n}\Delta_m^n D^Y(\chi')))+\\
&\quad \quad +\sqrt{\frac{n}{NT}}\sum_{m=1}^{ \left\lfloor nT \right\rfloor} \frac{1}{n}F_N\left(\frac{2\pi}{T} (t-t_{m-1}^n)\right)
(g(\sqrt{n}\Delta_m^n D^Y(\chi'))- g(\beta_m^n)).
\end{align*}
In view of Lemma~\cite[Lemma 2.2.11]{JP:12} it is sufficient to prove that 
\begin{align}
\sum_{m=1}^{ \left\lfloor nT \right\rfloor}\mathbb{E}\left[\|U_m^{n,1}\|^2\right] \to 0, \label{eq:quadconv}\\
\sum_{m=1}^{ \left\lfloor nT \right\rfloor}\mathbb{E}\left[U_m^{n,1}| \mathcal{F}_{t^n_{m-1}}\right] \stackrel{\mathbb{P}}{\to} 0,\label{eq:condexpect}
\end{align}
and
\begin{align}
\mathbb{P}\left[\left\|\sum_{m=1}^{ \left\lfloor nT \right\rfloor}U_m^{n,2}\right\| > \varepsilon \right] \to 0 \, , 
\label{eq:jumpsclt}
\end{align}
where 
\begin{align*}
U_m^{n,1}&=\sqrt{\frac{1}{nNT}}F_N\left(\frac{2\pi}{T} (t-t_{m-1}^n)\right)
(g(\sqrt{n}\Delta_m^n D^Y(\chi'))- g(\beta_m^n)) \, ,\\
U_m^{n,2}&=\sqrt{\frac{1}{nNT}}F_N\left(\frac{2\pi}{T} (t-t_{m-1}^n)\right)
(g(\sqrt{n}\Delta_m^n Y)- g(\sqrt{n}\Delta_m^n D^Y(\chi')).
\end{align*}
Let us first focus on $U_m^{n,1}$. By~\cite[Lemma 5.3 and Lemma 5.4]{BNGJPS:06}) and no fixed time of discontinuity of $X$ we have
\[
\sup_{m}\mathbb{E}\left[\left\|g(\sqrt{n}\Delta_m^n D^Y(\chi'))- g(\beta_m^n))\right\|^2\right] \to 0
\]
and we can therefore estimate~\eqref{eq:quadconv} by
\[
\sup_m\mathbb{E}\left[\left\|g(\sqrt{n}\Delta_m^n D(\chi')- g(\beta_m^n))\right\|^2\right] \frac{1}{T}\sum_{m=1}^{ \left\lfloor nT \right\rfloor}  \frac{1}{nN}F^2_N\left(\frac{2\pi}{T} (t-t_{m-1}^n)\right),
\]
which converges to $0$ due to Lemma~\ref{lem:Fejerconv}.

Concerning~\eqref{eq:condexpect}, it is possible to decompose  
\[
 g(\sqrt{n}\Delta_m^n D^Y(\chi'))- g(\beta_m^n)=A_m^n+B_m^n,
\]
where for all $m$, $\mathbb{E}\left[A_m^n| \mathcal{F}_{t^n_{m-1}}\right]=0$ and $\mathbb{E}\left[\|B_m^n\|\right] \leq\frac{1}{n^{\eta}}$ with $\eta >\frac{\gamma-1}{2 \gamma}$ (see~\cite[Section 5.3.3, C]{JP:12}.
Then
\begin{align*}
& \sqrt{\frac{n}{NT}}\sum_{m=1}^{ \left\lfloor nT \right\rfloor}\frac{1}{n}F_N\left(\frac{2\pi}{T} (t-t_{m-1}^n)\right)
\mathbb{E}\left[g(\sqrt{n}\Delta_m^n D^Y(\chi'))- g(\beta_m^n)|\mathcal{F}_{t^n_{m-1}}\right]\\
&\quad =\sqrt{\frac{n}{NT}}\sum_{m=1}^{ \left\lfloor nT \right\rfloor}\frac{1}{n}F_N\left(\frac{2\pi}{T} (t-t_{m-1}^n)\right)
\mathbb{E}\left[A_m^n+B_m^n|\mathcal{F}_{t^n_{m-1}}\right]\\
&\quad\leq C n^{\frac{\gamma-1}{2 \gamma}}\sup_m\mathbb{E}\left[\|B_m^n\|\right]
\sum_{m=1}^{ \left\lfloor nT \right\rfloor}\frac{1}{n}F_N\left(\frac{2\pi}{T} (t-t_{m-1}^n)\right),
\end{align*}
converges $0$ due to Lemma~\ref{lem:Fejerconv} and thus yields~\eqref{eq:condexpect}. 
Condition~\eqref{eq:jumpsclt} follows from the assumption $L(\eta)$ for $\eta \geq\frac{\gamma-1}{2 \gamma}$ and a similar estimate as in~\eqref{eq:Fourierconsjump}.

Let us now turn to $\sqrt{\frac{nT}{N}}E^{n,N}_2(t)$, which we write as
\begin{align*}
 &\sqrt{\frac{nT}{N}}E^{n,N}_2(t) =\sum_{m=1}^{ \left\lfloor nT \right\rfloor} Z_m^{n,N},
\end{align*}
where
\[
 Z^{n,N}_m=\sqrt{\frac{1}{nNT}} F_N\left(\frac{2\pi}{T}(t-t_{m-1}^n)\right)
 (g(\beta_m^n)-\rho_{m-1}^n(g)).
\]
Since $\mathbb{E}\left[(g(\beta_m^n)-\rho_{m-1}^n(g))|\mathcal{F}_{t_{m-1}^n}\right]=0$, we also have 
\[
\lim_{n,N \to \infty} \sum_{m=1}^{ \left\lfloor nT \right\rfloor} \mathbb{E}\left[Z^{n,N}_m|\mathcal{F}_{t_{m-1}^n}\right]=0.
\]
Moreover,
\begin{align*}
 &\mathbb{E}\left[Z^{n,N}_{m,ij}\overline Z^{n,N}_{m,i'j'}|\mathcal{F}_{t_{m-1}^n}\right]\\
 &\quad=\frac{1}{T}
 (\rho_{m-1}^n(g_{ij}g_{i'j'})-\rho_{m-1}^n(g_{ij})\rho_{m-1}^n(g_{i'j'}))\frac{1}{nN}F^2_N\left(\frac{2\pi}{T} (t-t_{m-1}^n)\right).
\end{align*}
Thus we have
\begin{align*}
&\sum_{m=1}^{ \left\lfloor nT \right\rfloor} \mathbb{E}\left[Z^{n,N}_{m,ij}\overline Z^{n,N}_{m,i'j'}|\mathcal{F}_{t_{m-1}^n}\right]\\
&\quad=\sum_{m=1}^{ \left\lfloor nT \right\rfloor}\frac{1}{T} (\rho_{m-1}^n(g_{ij}g_{i'j'})-\rho_{m-1}^n(g_{ij})\rho_{m-1}^n(g_{i'j'}))\frac{1}{n} \frac{F^2_N\left(\frac{2\pi}{T} (t-t_{m-1}^n)\right)}{N}.
\end{align*}
Due to Lemma~\ref{lem:Fejerconv} and Remark~\ref{rem:Fejerconv} the limit of this expression is given by
\begin{align*}
V^{iji'j'}_t:&=\lim_{N \to \infty}\frac{1}{T}\int_0^T(\rho_{g_{ij}g_{i'j'}}(X_s)-\rho_{g_{ij}}(X_s)\rho_{g_{i'j'}}(X_s))\frac{F^2_{ N}\left(\frac{2\pi}{T} (t-s)\right)}{N}ds\\
&=\frac{2}{3}(\rho_{g_{ij}g_{i'j'}}(X_t)-\rho_{g_{ij}}(X_t)\rho_{g_{i'j'}}(X_t)).
\end{align*}

In view of Theorem~\cite[Theorem IX.7.28]{JS:03} it remains to verify that
\begin{align}\label{eq:like4mom}
\sum_{m=1}^{ \left\lfloor nT \right\rfloor} \mathbb{E}\left[\|Z^{n,N}_m\|^2 1_{\{\|Z^{n,N}_m\| > \varepsilon\}}|\mathcal{F}_{t_{m-1}^n}\right] &\stackrel{\mathbb{P}}{\to}0
\end{align}
for all $\varepsilon >0$.
By the Cauchy-Schwarz inequality we have
\[
  \mathbb{E}\left[\|Z^{n,N}_m\|^2 1_{\{|Z^{n,N}_m| > \varepsilon\}}|\mathcal{F}_{t_{m-1}^n}\right]\leq 
   \sqrt{\mathbb{E}\left[\|Z^{n,N}_m\|^4 |\mathcal{F}_{t_{m-1}^n}\right]}\sqrt{\mathbb{E}\left[ 1_{\{|Z^{n,N}_m| > \varepsilon\}}|\mathcal{F}_{t_{m-1}^n}\right]}.
\]
By definition of $Z_m^{n,N}$ and the polynomial growth of $g$, we can further estimate
\begin{align*}
  & \sqrt{\mathbb{E}\left[\|Z^{n,N}_m\|^4 |\mathcal{F}_{t_{m-1}^n}\right]} \\
  =& \sqrt{\mathbb{E}\left[\frac{1}{T^2}\left\|g\left(\beta_m^n\right)-\rho_{m-1}^n(g)\right\|^4 \left(\frac{1}{nN}\right)^2
 F_N^4\left(\frac{2\pi}{T} (t-t_{m-1}^n)\right) |\mathcal{F}_{t_{m-1}^n}\right]}\\
 \leq & C \frac{1}{T}  \frac{1}{n}
\frac{ F_N^2\left(\frac{2\pi}{T} (t-t_{m-1}^n)\right)}{N}.
\end{align*}
Taking again the polynomial growth of $g$ into account, there exists some $p \geq 0$ such that
\[
 \|g\left(\beta_m^n\right)-\rho_{m-1}^n(g)\| \leq C (1+\|U_{m-1}^n\|^p), \quad  \mathbb{P}\textrm{-a.s.} \, ,
\]
where $U_{m-1}^n=\sqrt{n}\sqrt{X_{t_{m-1}^n}}\Delta_m^n Z~\sim N(0,X_{t_{m-1}^n})$.
Thus
 \begin{align*}
& \mathbb{E}\left[ 1_{\{\|Z^{n,N}_m\| > \varepsilon\}}|\mathcal{F}_{t_{m-1}^n}\right] \\
=& \mathbb{P}\left[ \|Z^{n,N}_m\| > \varepsilon |\mathcal{F}_{t_{m-1}^n}\right]\\
 \leq & \mathbb{P}\left[ C (1+\|U_{m-1}^n\|^p)\sqrt{ \frac{1}{nNT}}
 F_N\left(\frac{2\pi}{T} (t-t_{m-1}^n)\right) > \varepsilon|\mathcal{F}_{t_{m-1}^n}\right]\\
 = & \mathbb{P}\left[ \|U_{m-1}^n\| > \underbrace{\left(\frac{1}{C}\left( \sqrt{nNT} F^{-1}_N\left(\frac{2\pi}{T} (t-t_{m-1}^n)\right)\varepsilon-1\right)\right)^{\frac{1}{p}}}_{\to \infty \textrm{ as } n,N \to \infty}|\mathcal{F}_{t_{m-1}^n}\right].
 \end{align*}
Since this tends to $0$, we can estimate~\eqref{eq:like4mom} by
\begin{align*}
&\sum_{m=1}^{ \left\lfloor nT \right\rfloor}  C \frac{1}{T}  \frac{1}{n}
 \frac{F_N^2\left(\frac{2\pi}{T} (t-t_{m-1}^n)\right)}{N}\sqrt{\mathbb{P}\left[ \|Z^{n,N}_m\| > \varepsilon |\mathcal{F}_{t_{m-1}^n}\right]}\stackrel{\mathbb{P}}{\to}0,
\end{align*}
where convergence to $0$ follows from Lemma~\ref{lem:Fejerconv} and the above estimate for
\[
\mathbb{P}\left[ \|Z^{n,N}_m\| > \varepsilon |\mathcal{F}_{t_{m-1}^n}\right] \, ,
\]
hence Equation~\eqref{eq:like4mom} is verified.  Moreover, similarly as in the proof of~\cite[Proposition 4.1]{BNGJPS:06}, we have
\[
  \mathbb{E}\left[Z^{n,N}_m\Delta_m^n Z|\mathcal{F}_{t_{m-1}^n}\right]=0
\]
and
\[
  \mathbb{E}\left[Z^{n,N}_{m}\Delta_m^n M|\mathcal{F}_{t_{m-1}^n}\right]=0
\]
for any bounded martingale $M$ which is orthogonal to the Brownian motion $Z$. 
The assertion is now implied by all these estimates and~\cite[Theorem IX.7.28]{JS:03}. 
\end{proof}

\section{Covariance of Covariance Estimation - Proof of Theorem~\ref{th:covcov} }\label{sec:covcov}

In this section we prove Theorem~\ref{th:covcov}, i.e., a central limit theorem for the estimator of the integrated covariance of $X$ obtained from the reconstructed path $\widehat{X}^{n,N}$.

\begin{proof}
Let us decompose
\begin{equation}
\begin{split}\label{eq:decomp2step}
 \sqrt{m} \left( V(\widehat{X}^{n,N},f,0)_T^m-\int_0^T \rho_f(q_s) ds\right)&= \sqrt{m}\left( V(\widehat{X}^{n,N},f,0)_T^m-V(X,f,0)_T^m\right) + \\
  &+\sqrt{m}\left( V(X,f,0)_T^m-\int_0^T \rho_f(q_s) ds)\right).
 \end{split}
 \end{equation}
In view~\cite[Theorem 5.3.5 and 5.3.6]{JP:12}, the second term converges to the stated Gaussian random variable. 
Hence we only have to prove that the first term converges to $0$ in probability.
Due to the assumptions on $f$, it can be estimated by
\begin{align*}
 &\sqrt{m}\left\| V(\widehat{X}^{n,N},f,0)_T^m-V(X,f,0)_T^m\right\|\\
&\quad \leq
 \frac{1}{\sqrt{m}}\sum_{p=1}^{ \left\lfloor mT \right\rfloor}  \left\|f\left(\sqrt{m}\Delta_p^m\widehat{X}^{n,N}\right)
 -f\left(\sqrt{m}\Delta_p^m X\right)\right\| \\
 &\quad\leq  m^{\frac{\alpha-1}{2}}C\sum_{p=1}^{ \left\lfloor mT \right\rfloor} 
 \left\|\Delta_p^m\widehat{X}^{n,N}-\Delta_p^m X\right\|^{\alpha}\\
 &\quad\leq m^{\frac{\alpha-1}{2}}2C\sum_{p=0}^{ \left\lfloor mT \right\rfloor} 
 \left\|\widehat{X}_{t_p^m}^{n,N}- X_{t_p^m}\right\|^{\alpha}
\end{align*}
Denoting by $A_p^m$ the set 
\begin{align}\label{eq:contset}
A_p^m=\{\omega \,|\, t \mapsto X_t(\omega) \textrm{ is continuous in } [t_{p-1}^m, t_{p+1}^m]  \},\quad p=1,\ldots,\left\lfloor mT \right\rfloor,
\end{align}
we further split the above expression into
\begin{align*}
 m^{\frac{\alpha-1}{2}}\sum_{p=0}^{ \left\lfloor mT \right\rfloor}\left\|\widehat{X}_{t_p^m}^{n,N} -X_{t^m_p}\right\|^{\alpha}1_{A_p^m}
 + m^{\frac{\alpha-1}{2}}\sum_{p=0}^{ \left\lfloor mT \right\rfloor}\left\|\widehat{X}_{t_p^m}^{n,N} -X_{t^m_p}\right\|^{\alpha}1_{(A_p^m)^c}.
\end{align*}
Due to the assumption of finite activity jumps, the second sum contains a.s.~only finitely many summands and thus converges to $0$ a.s. since $\alpha <1$.
By Lemma~\ref{lem:supconvergence} below, the relation between $m$ and $n$ and the condition on $\iota$, the
first sum converges to $0$ in probability. 
\end{proof}

\begin{lemma}~\label{lem:supconvergence}
Let the conditions of Theorem~\ref{th:covcov} be in force (with possibly $\alpha=1$) and denote by 
$A_p^m$ the sets defined in~\eqref{eq:contset}.
Then
\begin{align}\label{eq:consitencyrate}
m^{\frac{\alpha-1}{2}}\sum_{p=0}^{ \left\lfloor mT \right\rfloor} 
 \left\|\widehat{X}_{t^m_p}^{n,N}- X_{t^m_p}\right\|^{\alpha}1_{A_p^m} \stackrel{\mathbb{P}}{\to}0.
\end{align}
\end{lemma}

\begin{proof}
By localizing we can assume that $X$ is uniformly bounded. In fact, consider a localizing sequence
\[
\tau_k=\inf\{t \geq 0\, |\, \|X_t\| \geq k\}, \quad k \in \mathbb{N},
\]
and the processes
\begin{align*}
Y(k)_t&=y+\int_0^t b^Y_s ds + \int_0^{t\wedge \tau_k}\sqrt{X_{s}} dZ_s + \int_0^t \int_{\mathbb{R}^d} \chi(\xi)  (\mu^Y(d\xi,ds)- K_{s}(d\xi)ds) \\
&\quad+\int_0^t \int_{\mathbb{R}^d} (\xi-\chi(\xi))  \mu^Y(d\xi,ds),\\
X(k)_t&=X_t1_{\{t \leq \tau_k\}},
\end{align*}
where $(X_t(k))_{t \geq 0}$ is uniformly bounded by definition.
Moreover, define
\begin{align*}
\widehat{\rho_{g}(X(k))}^{n,N}_{t}:=\frac{1}{T}\sum_{j=-N}^{N} \left(1-\frac{|j|}{N}\right)e^{\im \frac{2\pi}{T} jt}V(Y(k),g,j)_T^{n}
\end{align*}
and
\begin{align*}
\widehat{X(k)}_t^{n,N}:=\rho^{-1}_g\left(\widehat{\rho_{g}(X(k))}^{n,N}_{t}\right).
\end{align*}
Then the left hand side of
\begin{align*}
\mathbb{P}\left[m^{\frac{\alpha-1}{2}}\sum_{p=0}^{ \left\lfloor mT \right\rfloor} 
 \left\|\widehat{X}_{t^m_p}^{n,N}- X_{t^m_p}\right\|^{\alpha} >\varepsilon\right]&\leq \mathbb{P}\left[\tau_k \leq T\right]\\
 &+
 \mathbb{P}\left[m^{\frac{\alpha-1}{2}}\sum_{p=0}^{ \left\lfloor mT \right\rfloor} 
 \left\|\widehat{X(k)}_{t^m_p}^{n,N}- X(k)_{t^m_p}\right\|^{\alpha}  >\varepsilon\right]
\end{align*}
tends to $0$, if the second term on the right hand side does. Therefore, we can assume uniform boundedness of $X$.

By the mean value theorem we obtain the identity
\begin{align*}
 \|\widehat{X}_{t^m_p}^{n,N}-X_{t^m_p}\|^{\alpha}= \left\|(\nabla \rho_g(\zeta^{n,N}_{t^m_p}))^{-1}\left(\widehat{\rho_{g}(X)}^{n,N}_{t^m_p}-\rho_g(X_{t^m_p})\right)\right\|^{\alpha},
\end{align*}
where $ \zeta^{n,N}_{t_p}$ is a random variable satisfying $\|\zeta^{n,N}_{t_p}-X_{t_p}\| \leq \|\widehat{X}_{t_p}^{n,N}-X_{t_p}\|$.
Due to the continuity assumption on $x\mapsto(\nabla \rho_g(x))^{-1}$ and boundedness of $X$, ~\eqref{eq:consitencyrate} converges to $0$ in probability if
\[
m^{\frac{\alpha-1}{2}}\sum_{p=0}^{ \left\lfloor mT \right\rfloor} 
\left\|\widehat{\rho_{g}(X)}^{n,N}_{t^m_p}-\rho_g(X_{t^m_p})\right\|^{\alpha}1_{A_p^m}\stackrel{\mathbb{P}}{\to}0.
\]

An inspection of the proof of Theorem~\ref{th:Fourierestconv} reveals that this is the case if the conditions between $m$, $n$ and $N$ are satisfied. Indeed, we split $\widehat{\rho_{g}(X)}^{n,N}_{t^m_p}-\rho_g(X_{t^m_p})$ in the same parts as in Theorem~\ref{th:Fourierestconv}, that is,
\[
\widehat{\rho_{g}(X)}^{n,N}_{t^m_p}-\rho_g(X_{t^m_p})=\widetilde{E}_0^N(t^m_p)+\sum_{i=1}^3 E^{n,N}_i(t^m_p),
\]
where $E^{n,N}_i(t)$, $i=1,2,3$ are defined in~\eqref{eq:E1} - \eqref{eq:E3} and $\widetilde{E}_0^N(t)$ is here given by 
\[
\widetilde{E}_0^N(t)=\sum_{k=-N}^{N} \left(1-\frac{|k|}{N}\right)e^{\im \frac{2\pi}{T} kt}\mathcal{F}(\rho_g(X))(k)-\rho_g(X_t).
\]

We start by showing 
\[
m^{\frac{\alpha-1}{2}}\sum_{p=1}^{ \left\lfloor mT \right\rfloor}\|\widetilde{E}^{N}_0(t^m_p)\|^{\alpha}1_{A_p^m}
\stackrel{a.s.}{\to} 0.
\]
By the assumption of finitely many jumps and $\delta$-H\"older continuity between two jumps, we can find a uniform  (in $m$) bound 
$C_{\omega}$ (depending on $\omega$) such that
\[
\sup_p\|\widetilde{E}^N_0(t^m_p)(\omega)\|1_{A_p^m} \leq C_{\omega} N^{-\delta}.
\]
Indeed, this is due to the fact that no jump occurs in $[t_{p-1}^m,t_{p+1}^m]$, the condition on $m$ and $N$, namely $m=LN^{\kappa}$ for some constant $L$ and $\kappa <1$,  and the way how the F\'ejer kernel declines, in particular that
\[
\int_{\frac{2\pi}{N+1}}^\pi F_N(x)dx\leq \frac{C}{N}
\]
holds true.
Hence we obtain 
\[
m^{\frac{\alpha-1}{2}}\sum_{p=1}^{ \left\lfloor mT \right\rfloor}\|\widetilde{E}^N_0(t^m_p)(\omega)\|^{\alpha}\leq C_{\omega} m^{\frac{\alpha+1}{2}}N^{-\delta \alpha},
\]
which tends to $0$ due to the relation of $m$ and $N$. 

Similarly we have a uniform (in $m$) convergence rate for $\sup_p \|E^{n,N}_3(t^m_p)\|1_{A_p^m}$ to $0$ which is of order $\max(N^{-(\gamma-1)}, n^{-\delta})$. The same arguments thus yield
\[
m^{\frac{\alpha-1}{2}}\sum_{p=1}^{ \left\lfloor mT \right\rfloor}\|E^{n,N}_3(t^m_p)\|^{\alpha}1_{A_p^m}
\stackrel{a.s.}{\to} 0.
\]

Concerning
\[
m^{\frac{\alpha-1}{2}}\sum_{p=1}^{ \left\lfloor mT \right\rfloor}\|E^{n,N}_1(t^m_p)\|^{\alpha}1_{A_p^m}
\stackrel{\mathbb{P}}{\to} 0,
\]
 it suffices to show that
\[
m^{\frac{\alpha+1}{2\alpha}}\sup_{t}\|E^{n,N}_1(t)\|\stackrel{\mathbb{P}}{\to} 0
\]
By the relation of $m$ and $n$, this then follows from the fact that 
\[
n^{\theta}\sup_{t}\|E^{n,N}_1(t)\|\stackrel{\mathbb{P}}{\to} 0
\]
for all $\theta \leq \frac{\gamma-1}{2\gamma}$. This latter property then follows from uniform convergence (in $t$) of 
\[
\sum_{i=1}^{ \left\lfloor nT \right\rfloor}\frac{1}{n} F_N\left(\frac{2\pi}{T}(t-t_{i-1}^n)\right)\quad \textrm { and } \quad \sum_{i=1}^{ \left\lfloor nT \right\rfloor}\frac{1}{n} \frac{F^2_N\left(\frac{2\pi}{T}(t-t_{i-1}^n)\right)}{N}.
\] 

In order to prove
\begin{align*}
&m^{\frac{\alpha-1}{2}}\sum_{p=1}^{ \left\lfloor mT \right\rfloor}\|E^{n,N}_2(t^m_p)\|^{\alpha}1_{A_p^m}\leq m^{\frac{\alpha-1}{2}}\sum_{p=1}^{ \left\lfloor mT \right\rfloor}\|E^{n,N}_2(t^m_p)\|^{\alpha}\\
&\quad=m^{\frac{\alpha-1}{2}}\sum_{p=1}^{ \left\lfloor mT \right\rfloor} 
\left\|\frac{1}{T}\sum_{i=1}^{ \left\lfloor nT \right\rfloor} \frac{1}{n} F_N\left(\frac{2\pi}{T}(t_p^m-t_{i-1}^n)\right)(g(\beta_i^n)-\rho_{i-1}^n(g))\right\|^{\alpha}\stackrel{\mathbb{P}}{\to}0,
\end{align*}
we estimate the $L_1$-norm of this expression by
\begin{align*}
&m^{\frac{\alpha+1}{2}}\sup_{t_p^m}\mathbb{E}\left[\left\|\frac{1}{T}\sum_{i=1}^{ \left\lfloor nT \right\rfloor} \frac{1}{n} F_N\left(\frac{2\pi}{T}(t_p^m-t_{i-1}^n)\right)(g(\beta_i^n)-\rho_{i-1}^n(g))\right\|^{2}\right]^{\frac{\alpha}{2}}\\
&\quad \leq m^{\frac{\alpha+1}{2}}\mathbb{E}\left[\sup_i\|(g(\beta_i^n)-\rho_{i-1}^n(g))\|^2\right]^{\frac{\alpha}{2}}\sup_{t_p^m}
\left\|\frac{1}{T^2}\sum_{i=1}^{ \left\lfloor nT \right\rfloor} \frac{1}{n^2} F^2_N\left(\frac{2\pi}{T}(t_p^m-t_{i-1}^n)\right)\right\|^{\frac{\alpha}{2}}\\
&\quad \leq  C m^{\frac{\alpha+1}{2}}n^{\frac{(1-\gamma)\alpha}{2\gamma}},
\end{align*}
which converges to $0$ due to the relation between $m$ and $n$. The last inequality is a consequence of Lemma~\ref{lem:Fejerconv}, where the assertion of~\eqref{eq:FN2} can be extended to uniform convergence.

\end{proof}

\begin{remark}
The reason for the assumption $\alpha <1 $ in the assumptions of Theorem~\ref{th:covcov} comes from the requirements of~\cite[Theorem 5.3.5 c) and 5.3.6]{JP:12} in the case of jumps. 
If $X$ has continuous trajectories, the result also holds true for $\alpha=1$.
\end{remark}

\appendix

\section{Simulation results}\label{sec:sim}

In this section we illustrate our theoretical results in the case of a multivariate affine model, where both the log-price $Y$ and the instantaneous covariance process $X$ can jump. More precisely, we consider a multivariate Bates-type model (compare, e.g.,~\cite{cfmt, fonsecaetal1, leippoldtrojani}, of the form
\begin{align*}
Y_t&=y+\int_0^t b_s ds + \int_0^t\sqrt{X_{s-}} dZ_s +\int_0^t \int_{\mathbb{R}^d} \xi  \mu^Y(d\xi,ds) ,\\
X_t&=x+\int_0^t(b+MX_t+X_tM^{\top})dt+\sqrt{X_t}dB_t\Sigma+\Sigma dB_t^{\top}\sqrt{X_t}+\\
&\quad+\int_0^t \int_{S^+_d} \xi  \mu^X(d\xi,ds) \, ,
\end{align*}
where
\begin{itemize}
\item $Z$ is a $d$-dimensional Brownian motion correlated with the $d \times d$ matrix of Brownian motions $B$ such that $Z=\sqrt{1-\rho^{\top}\rho}W+B\rho$, where $\rho \in [-1,1]^d$ such that $\rho^{\top}\rho\leq 1$ and $W$ is a $d$-dimensional Brownian motion independent of $B$,
\item $\mu^Y(d\xi,dt)$ is the random measure associated with the jumps of $Y$, whose compensator is given by 
$\sum_{i=1}^d\lambda^{Y_i} F^{Y_i}(d\xi_i)dt$, where $\lambda^{Y_i} > 0$ and $F^{Y_i}$ denotes the Gaussian density with mean $\mu_i$ and standard deviation $\sigma_i$,
\item $\mu^X(d\xi,dt)$ is the random measure associated with the jumps of $X$, whose compensator is given by $\lambda^{X_{11}}F^{X_{11}}(d\xi_{11})dt$, where $\lambda^{X_{11}}>0$ and $F^{X_{11}}$ denotes the density of the exponential distribution with parameter $\theta$,\footnote{We here only suppose that $X_{11}$ can jump.}
\item the drift of $Y$ is given by $b_{s,i}=-\frac{1}{2}X_{s,ii}-\lambda^{Y_i}(e^{\mu_i-\frac{1}{2}\sigma_i^2}-1)$ and 
\item the parameters of $X$ satisfy $M \in \mathbb{R}^{d \times d}$, $\Sigma \in S_d^+$, $b \in S_d^+$ such that 
\[
b-(d-1)\Sigma^2 \in S_d^+.
\]
\end{itemize}
Note that the truncation function of $Y$ is here chosen to be $0$.

As described in Section~\ref{sec:1} and Section~\ref{sec:4} we aim to recover the instantaneous covariance process $X$ and the parameters 
$\alpha:=\Sigma^2$ and $\rho$ from observations of $Y$. In order to be in accordance with market specifications, we simulate $Y$ and $X$ on $n=127750=511*250$ grid points, which corresponds to 1 year ($T=1$) of 1-minute data.
For our numerical simulation, we consider the case $d=2$ and use the following parameter values:

\begin{center}
\begin{tabular}{ | c | c |}
\hline
$\left(Y_{0,1}, Y_{0,2}\right)$ &$\left(0, 0\right)$\\ \hline
$\left(\begin{array}{c c}X_{0,11} & X_{0,12}\\
X_{0,12} & X_{0,22}\end{array}\right)$& $\left(\begin{array}{c c}0.09 & -0.036\\
-0.036 & 0.09
\end{array}\right)$ \\ \hline
$M$& $\left(\begin{array}{c c}-1.6 & -0.2\\
-0.4 & -1
\end{array}\right)$\\ \hline
$\alpha=\Sigma^2$ &  $\left(\begin{array}{c c}0.0725 & 0.06\\
0.06 & 0.1325
\end{array}\right)$\\ \hline
$b$ & $3.5 \alpha$\\ \hline
$\rho$ &$\left(-0.3 , -0.5\right)$\\ \hline
$\left(\lambda^{Y_1} ,\lambda^{Y_2}\right)$&$(100,100)$\\ \hline
$\left(\mu_1, \mu_2\right)$&$\left(-0.005,-0.003\right)$\\ \hline
$\left(\sigma_1,\sigma_2\right)$&$\left(0.015 ,0.02\right)$\\ \hline 
$\lambda^{X_{11}}$ & $10$\\ \hline
$\theta$ & $0.05$ \\\hline
\end{tabular}
\end{center}

In order to illustrate in particular that our estimator is robust to small and frequent jumps, the jump intensity of both log-prices is chosen to be quite high. Figure~\ref{Fig:log} below show simulated trajectories of the log-price and the instantaneous covariance process, where the jumps are removed in the second graph in each case.

\begin{figure}[ht]
\centering
\includegraphics[width=220pt,height= 150pt]{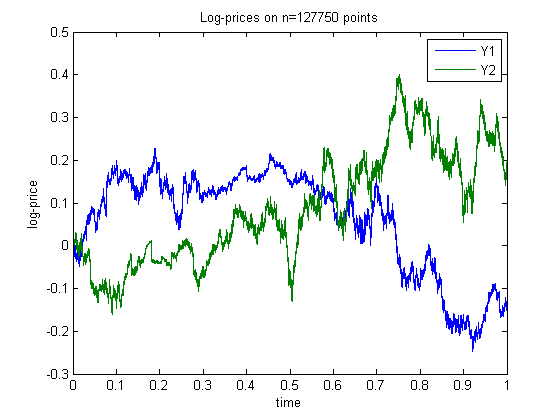}
\includegraphics[width=220pt,height= 150pt]{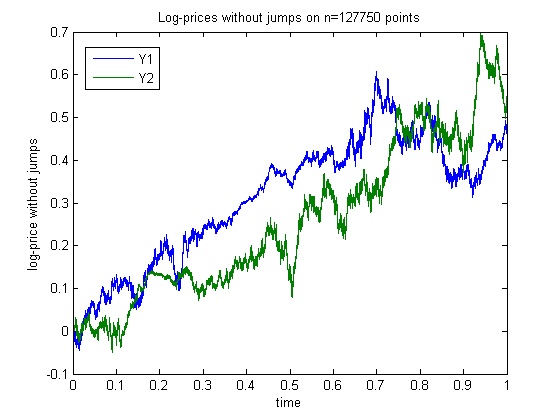}
\includegraphics[width=220pt,height= 150pt]{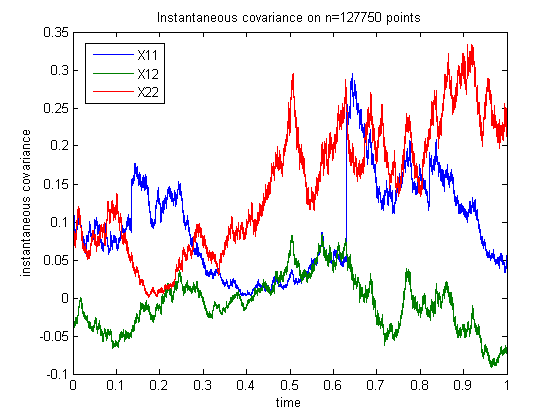}
\includegraphics[width=220pt,height= 150pt]{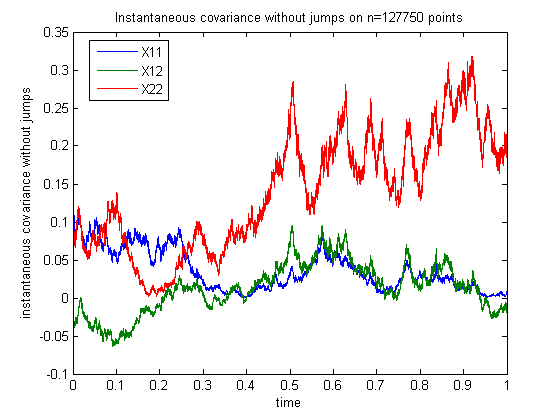}
\caption{Simulated log-price and instantaneous covariance with and without jumps on $n=127750$ points}\label{Fig:log}
\end{figure}

A comparison between the reconstructed and simulated trajectories of the instantaneous covariance process is shown in Figure~\ref{Fig:rec}. These figures illustrate that -- even in the case of (frequent) jumps in the log-price and in the variance (as it is the case for $X_{11}$) -- the paths of $X$ can be recovered very well.
For the reconstruction of the trajectories of $X$ we choose $N=210$ Fourier coefficients, which corresponds to the choice $\gamma\approx 2$ and $K\approx3$, as specified in Theorem~\ref{th:Fourierestconv}. This is a reasonable choice in view of an acceptable bias and a rather small variance. 
Both, the simulated as well as the reconstructed trajectories are evaluated at $2N+1$ points.
In our concrete implementation the estimator for the Fourier coefficients~\eqref{eq:Fouriercofest} is based on the Tauchen-Todorov specification of the function $g$, that is,
\[
g : \mathbb{R}^2 \to S_2, \quad  (y_1,y_2) \mapsto (\cos(y_i + 1_{\{j \neq i\}}y_j))_{i,j\in\{1,2\}}.
\]
In this case
\[
\rho_g(x)=\left( \begin{array}{c c}
e^{-\frac{1}{2} x_{11}} & e^{-\frac{1}{2} (x_{11}+2 x_{12}+x_{22})} \\
e^{-\frac{1}{2} (x_{11}+2 x_{12}+x_{22})}  &e^{-\frac{1}{2} x_{22}} \end{array}\right)
\]
and $\widehat{X}_t^{n,N}$ is obtained by
\begin{align*}
\widehat{X}_{t,ii}^{n,N}&=-2 \log \left(\widehat{\rho_{g_{ii}}(X)}^{n,N}_{t}\right), \quad i \in \{1,2\},\\
\widehat{X}_{t,12}^{n,N}&=\frac{1}{2}\left(-2 \log \left(\widehat{\rho_{g_{12}}(X)}^{n,N}_{t}\right)- \widehat{X}_{t,11}^{n,N}-\widehat{X}_{t,22}^{n,N}\right).
\end{align*}

The reconstructed trajectories of $X$ are then used to estimate the parameters $\alpha$ and $\rho$. To this end, we use 
the power variation estimators, i.e.,
\begin{align*}
V(\widehat{X}_{ij}^{n,N},f_r,0)_1^m&:=\frac{1}{m} \sum_{p=1}^m |\sqrt{m}\Delta_m^p \widehat{X}_{ij}^{n,N}|^{r}, \quad i,j \in \{1,2\},\\
V(\widehat{X}_{ii}^{n,N},Y_i,f_{r,s},0)_1^m&:=\frac{1}{m} \sum_{p=1}^m |\sqrt{m}\Delta_m^p \widehat{X}_{ii}^{n,N}|^r |\sqrt{m}\Delta_m^p Y_i|^{s}, \quad i \in \{1,2\},
\end{align*}
where
\begin{align*}
&f_r : \mathbb{R} \to \mathbb{R}_+, \quad  x \mapsto |x|^r,\\
&f_{r,s} : \mathbb{R}^2 \to \mathbb{R}_+, \quad  (x,y) \mapsto |x|^r|y|^s.
\end{align*}
These quantities are estimators for the power (co)variation of $X$ and $Y$. Indeed we have under the assumptions of Theorem~\ref{th:covcov} 
\begin{align*}
&V(\widehat{X}_{ii}^{n,N},f_r,0)_1^m \to \int_0^1 \rho_{f_r}(4\alpha_{ii} X_{s,ii}) ds \\
&\qquad=\sqrt{\frac{1}{\pi}}2^{\frac{r}{2}}\Gamma\left(\frac{r+1}{2}\right)(4\alpha_{ii})^{\frac{r}{2}}\int_0^1 X_{s,ii}^{\frac{r}{2}} ds,\\
&V(\widehat{X}_{12}^{n,N},f_r,0)_1^m \to \int_0^1 \rho_{f_r}(\alpha_{11} X_{s,11}+2\alpha_{12} X_{s,12}+\alpha_{22} X_{s,22}) ds\\ &\qquad=\sqrt{\frac{1}{\pi}}2^{\frac{r}{2}}\Gamma\left(\frac{r+1}{2}\right) \int_0^1 (\alpha_{11} X_{s,11}+2\alpha_{12} X_{s,12}+\alpha_{22} X_{s,22})^{\frac{r}{2}} ds\\
&\qquad=:PV_{12}(X,\alpha_{11}, \alpha_{12}, \alpha_{22}),\\
&V(\widehat{X}_{ii}^{n,N},Y_i,f_{r,s},0)_1^m \to \int_0^1 \rho_{f_{r,s}} 
\left(\left(\begin{array}{c c} 4 \alpha_{ii} X_{u,ii} &2(\sqrt{\alpha}\rho)_i X_{u,ii} \\
2(\sqrt{\alpha}\rho)_i X_{u,ii} & X_{u,ii}\end{array}\right)\right) du\\
&\qquad= \frac{1}{\pi}2^{\frac{r+s}{2}}\Gamma\left(\frac{r+1}{2}\right)\Gamma\left(\frac{s+1}{2}\right)  \\
&\qquad\quad \times{}_2F_1\left(-\frac{r}{2},-\frac{s}{2};\frac{1}{2}; \left(\frac{(\sqrt{\alpha}\rho)_i}{\sqrt{\alpha_{ii}}}\right)^2\right)
 (4 \alpha_{ii})^{\frac{r}{2}}\int_0^1  X_{u,ii}^{\frac{r+s}{2}}du\\
&\qquad=:PC_{ii}(X,Y,\alpha,\rho)
\end{align*}
as $m,n,N \to \infty$. The formulas on the right hand sides follow from the expressions for the absolute moments of the bivariate Gaussian distribution (see, e.g.,\cite{N:51}) and ${}_2F_1\left(a,b;c;x\right)$ denotes the Gaussian hypergeometric function. 

\begin{figure}[ht]
\centering
\includegraphics[width=240pt,height= 160pt]{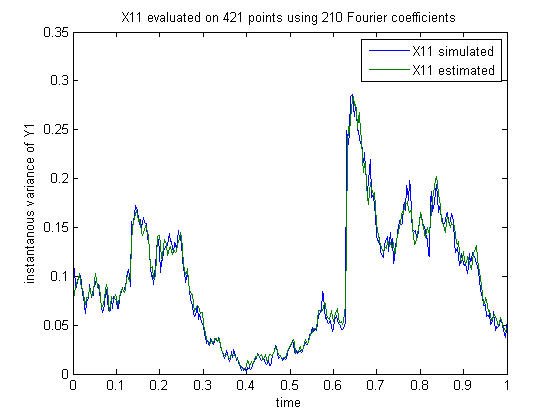}
\includegraphics[width=240pt,height= 160pt]{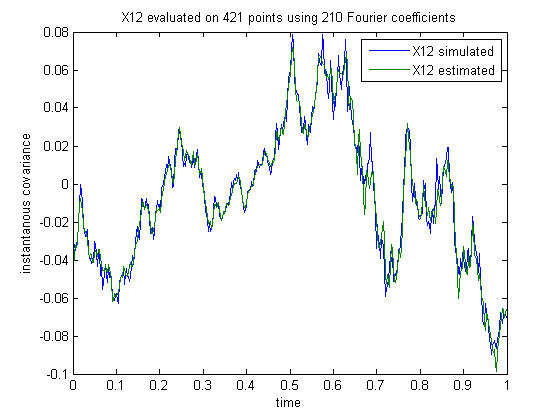}
\includegraphics[width=240pt,height= 160pt]{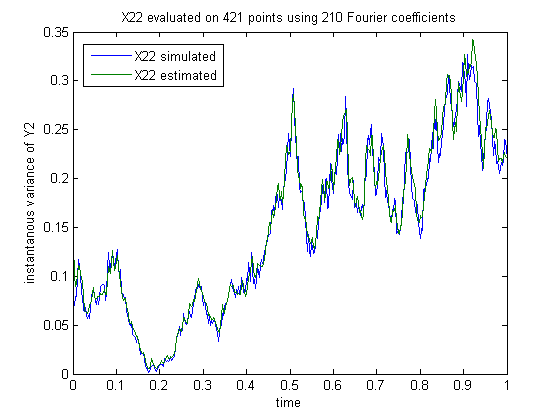}
\caption{Reconstructed and simulated instantaneous covariance evaluated on $421$ points}\label{Fig:rec}
\end{figure}

The estimators for $\alpha$ and $\rho$ can now be constructed via
\begin{align*}
\widehat{\alpha}_{ii}&=\frac{1}{8}\left(\frac{V(\widehat{X}_{ii}^{n,N},r,0)_1^m}{\sqrt{\frac{1}{\pi}}\Gamma\left(\frac{r+1}{2}\right)\frac{1}{m}\sum_{p=1}^m \left(\widehat{X}^{n,N}_{t_p^m,ii}\right)^{\frac{r}{2}}} ds\right)^{\frac{2}{r}}\\
\widehat{\alpha}_{12}&=\underset{ \alpha_{12} \in [-\sqrt{\widehat{\alpha}_{11}\widehat{\alpha}_{22}}, \sqrt{\widehat{\alpha}_{11}\widehat{\alpha}_{22}}]}{\operatorname{argmin}}\left(V(\widehat{X}_{12}^{n,N},f_r,0)_1^m-
PV_{12}(\widehat{X}^{n,N},\widehat{\alpha}_{11}, \alpha_{12}, \widehat{\alpha}_{22})\right)^2\\
\widehat{\rho}&=\underset{ \rho \in [-1,1]^2,\, \rho^{\top}{\rho}\leq 1}{\operatorname{argmin}}
\sum_{i=1}^2 \left(V(\widehat{X}_{ii}^{n,N},Y_i,f_{r,s},0)_1^m-PC_{ii}(\widehat{X}^{n,N},Y,\widehat{\alpha},\rho)\right)^2,
\end{align*}
where we discretize the corresponding integrals in $PV_{12}$ and $PC_{ii}$ and $\widehat{X}^{n,N}$ instead of $X$.
In our simulation study, we choose $r=\frac{1}{4}$ in $V(\widehat{X}_{11}^{n,N},f_r,0)_1^m$ and $r=1$ in $V(\widehat{X}_{ij}^{n,N},f_r,0)_1^m$ for $(ij)=(12)$ and $(ij)=(22)$ respectively. This is due to the fact that $X_{11}$ exhibits jumps and taking a lower power reduces the contribution of jumps in the power variation. In $V(\widehat{X}_{ii}^{n,N},Y_i,f_{r,s},0)_1^m$, $r$ and $s$ are chosen to be $\frac{1}{2}$. Figure~\ref{Fig:paramalpha} and~\ref{Fig:paramrho} show the estimated values for $\alpha$ and $\rho$ as a function of the grid points $m$. As a consequence of Theorem~\ref{th:covcov}, the grid corresponding to $m$ has to be coarsened  considerably with respect to the initial gridding with $n$ points (of order $n^{\frac{1}{4}}$ or even more depending on the power used). For this reason the number of grid points shown in the graphs is rather small. Nevertheless the estimation results are good approximations of the true parameter values and can further be improved by increasing $n$ and thus in turn also $m$.
\begin{figure}[ht]
\centering
\includegraphics[width=240pt,height= 160pt]{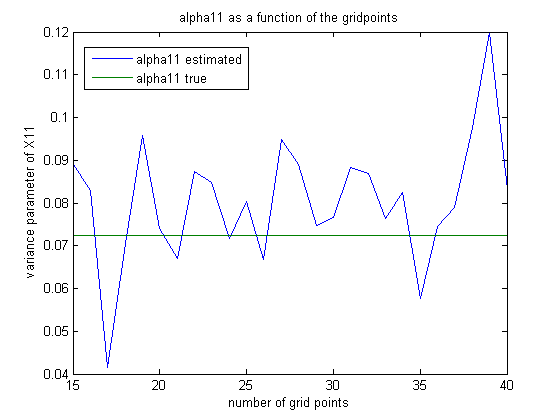}
\includegraphics[width=240pt,height= 160pt]{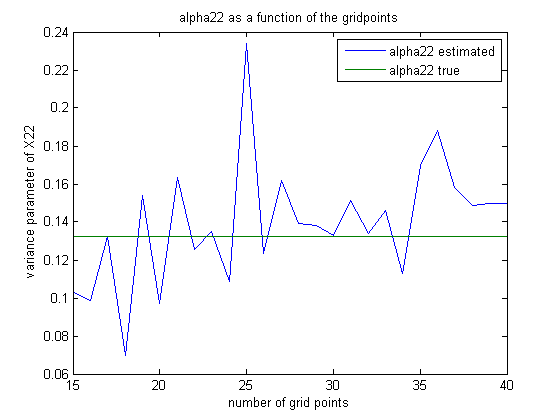}
\includegraphics[width=240pt,height= 160pt]{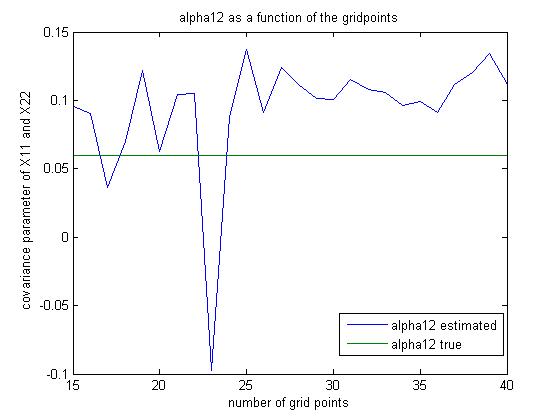}
\caption{Estimation of $\alpha$ as a function of the grid points}\label{Fig:paramalpha}
\end{figure}

\begin{figure}[ht]
\centering
\includegraphics[width=240pt,height= 160pt]{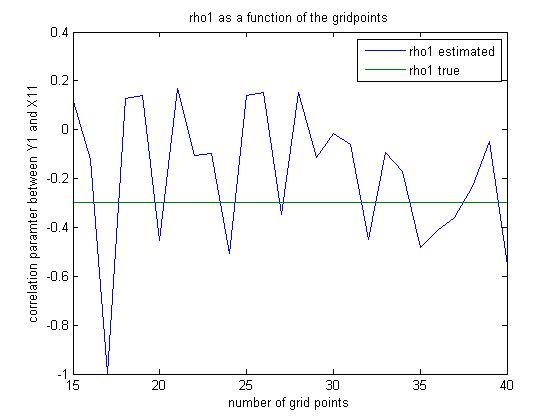}
\includegraphics[width=240pt,height= 160pt]{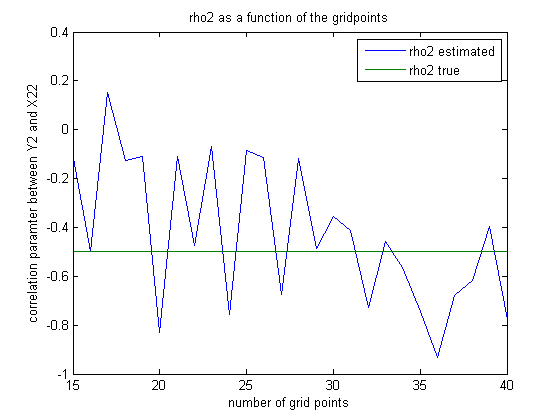}
\caption{Estimation of $\rho$ as a function of the grid points}\label{Fig:paramrho}
\end{figure}


\end{document}